\documentclass[twoside,11pt,reqno]{amsart}
\usepackage{amsmath,amssymb,amscd,mathrsfs,epic}
\usepackage{latexsym,amsthm,amscd}

\makeatletter
\@addtoreset{equation}{section}
\@addtoreset{figure}{section}

\newtheorem{Proposition}{Proposition}[section]
\newtheorem{Lemma}[Proposition]{Lemma}
\newtheorem{Theorem}[Proposition]{Theorem}
\newtheorem{Corollary}[Proposition]{Corollary}

\newtheorem{Remark}[Proposition]{Remark}

\newbox\squ  
\setbox\squ=\hbox{\vrule width.6pt
   \vbox{\hrule height.6pt width.4em\kern1ex
         \hrule height.6pt}%
   \vrule width.6pt}

\def\Tab{\mathscr{T}}

\def\ab{\operatorname{ab}}
\def\ev{\operatorname{ev}}

\def\ch{\operatorname{ch}}

\def\lmod{\operatorname{\!-mod}}
\def\rmod{\operatorname{mod-\!}}

\def\ev{{\operatorname{ev}}}

\def\C{{\mathbb C}}
\def\Q{{\mathbb Q}}
\def\Z{{\mathbb Z}}

\def\hom{{\operatorname{Hom}}}
\def\ann{{\operatorname{Ann}}}

\def\Prim{{\operatorname{Prim}U(\mathfrak{g})}}
\def\lPrim{{\operatorname{Prim}_\lambda U(\mathfrak{g})}}

\def\grk{{\operatorname{rk}\:}}
\def\rk{{\operatorname{rk}}}
\def\gkdim{{\operatorname{gkdim}\:}}

\def\bz{\hbox{\boldmath{$0$}}}
\def\eps{{\varepsilon}}

\def\phi{{\varphi}}
\def\phi{{\varphi}}
\def\row{\operatorname{row}}
\def\col{\operatorname{col}}

\def\la{{\lambda}}
\def\La{{\Lambda}}

\def\VA{\mathcal{V\!A}}

\def\@underbar#1{\settowidth{\@tempdimb}{#1}\@tempdimb=0.8\@tempdimb
                   \ooalign{#1\crcr
                         \hfil\rule[-.5mm]{\@tempdimb}{.4pt}\hfil}}

\newdimen\hoogte    \hoogte=8pt    
\newdimen\breedte   \breedte=8pt   
\newdimen\dikte     \dikte=0.4pt    

\newenvironment{young}{\begingroup
       \def\vr{\vrule height0.8\hoogte width\dikte depth 0.2\hoogte}
       \def\fbox##1{\vbox{\offinterlineskip
                    \hrule height\dikte
                    \hbox to \breedte{\vr\hfill##1\hfill\vr}
                    \hrule height\dikte}}
       \vbox\bgroup \offinterlineskip \tabskip=-\dikte \lineskip=-\dikte
            \halign\bgroup &\fbox{##\unskip}\unskip  \crcr }
       {\egroup\egroup\endgroup}
\def\diagram#1{\relax\ifmmode\vcenter{\,\begin{young}#1\end{young}\,}\else%
              $\vcenter{\,\begin{young}#1\end{young}\,}$\fi}

\begin{document}

\title[{\rm M\oe glin and Goldie}]{\boldmath M\oe glin's theorem and 
Goldie rank polynomials in Cartan type $A$}

\author{Jonathan Brundan}

\address{Department of Mathematics, University of Oregon, Eugene, OR 97403, USA}
\email{brundan@uoregon.edu}

\thanks{Research supported in part by NSF grant no. DMS-0654147}

\begin{abstract}
We use the theory of finite $W$-algebras associated to nilpotent orbits in the Lie
algebra
$\mathfrak{g} = \mathfrak{gl}_N(\C)$ to
give another proof of M\oe glin's 
theorem about completely prime
primitive ideals in the enveloping algebra $U(\mathfrak{g})$.
We also make some new observations about Joseph's Goldie rank
polynomials in Cartan type $A$.
\end{abstract}

\maketitle

\section{Introduction}

The space 
$\Prim$ of primitive ideals in the universal enveloping algebra
of the Lie algebra $\mathfrak{g} := \mathfrak{gl}_N(\C)$ has 
an unbelievably rich structure
which has been studied intensively since the 1970s.
In this article we revisit several of the foundational results about $\Prim$
from the perspective of the theory
of finite $W$-algebras that has been developed in the last few years
by Premet \cite{Pslice, Pslodowy, Pfinite, Pabelian, Pnew}, 
Losev \cite{Lsymplectic, Lclass, LcatO, L1D} 
and others \cite{BrG,BG,BGK,BKshifted,BKrep,GG,Gin}.
This article
was inspired by the most recent breakthrough of Premet in \cite{Pnew}, so
we start by discussing that in more detail.

Given a nilpotent element $e \in \mathfrak{g}$
there is associated a finite $W$-algebra $U(\mathfrak{g},e)$, and
Skryabin
proved that the category of
$U(\mathfrak{g},e)$-modules is equivalent to a certain category of generalized
Whittaker modules for $\mathfrak{g}$; see \cite{Pslice,Skryabin}. If $L$ is any irreducible $U(\mathfrak{g},e)$-module,
we define $I(L) \in \Prim$ by applying Skryabin's
equivalence of categories to get an irreducible $\mathfrak{g}$-module, 
then taking the annihilator of that module.
Premet's theorem \cite[Theorem B]{Pnew}
can be stated for $\mathfrak{g} = \mathfrak{gl}_N(\C)$
as follows.

\begin{Theorem}[Premet]\label{pti}
If 
$L$ is a finite dimensional irreducible $U(\mathfrak{g},e)$-module
and $I := I(L) \in \Prim$, then
the Goldie rank of $U(\mathfrak{g}) / I$ is equal to the dimension of $L$.
\end{Theorem}

Premet actually works with the finite $W$-algebra 
attached to a nilpotent element in an arbitrary reductive
Lie algebra, showing in analogous notation in
that general context
that $\grk U(\mathfrak{g}) / I$ always {\em divides} $\dim L$,
with equality if the Goldie field of 
$U(\mathfrak{g}) / I$ is isomorphic to the ring
of fractions of
a Weyl algebra.
The fact that this
condition for equality 
is satisfied for all $I \in \Prim$
when $\mathfrak{g} = \mathfrak{gl}_N(\C)$
follows from a result of Joseph from
\cite[$\S$10.3]{Jkos}.
A key step in Joseph's proof
involved showing in \cite[$\S$9.1]{Jkos}
that 
the ring of fractions of $U(\mathfrak{g}) / \ann_{U(\mathfrak{g})} M$
is isomorphic to the ring of fractions of $\mathscr L(M,M)$, the
$\operatorname{ad}\mathfrak{g}$-locally finite maps from $M$ to
itself, for all irreducible highest weight modules $M$. This is the weak
form of Kostant's problem;
see also \cite[12.13]{Je}.

In the same article, Joseph proved an additivity principle for certain
Goldie ranks which, when combined with the solution of the
weak form of Kostant's problem just mentioned,
led Joseph to the discovery of a systematic method for computing the
Goldie ranks of all 
primitive quotients of enveloping algebras in Cartan type $A$; 
see \cite[$\S$8.1]{Jkos}. Soon afterwards in \cite[$\S$5.1]{Jgoldie}, 
Joseph worked out a general
approach to compute Goldie ranks of primitive quotients in arbitrary
Cartan types via his remarkable theory of 
{Goldie rank polynomials}. These polynomials involve some mysterious
constants which even today
are only determined explicitly in Cartan type $A$; see the discussion
in \cite[$\S$1.5]{JIII} and use \cite[Lemma 5.1]{Jcyclic} to treat
Cartan type $A$.
Much more recently, 
in \cite[$\S$8.5]{BKrep}, we described a method for
computing the
dimensions of all 
finite dimensional irreducible representations of finite
$W$-algebras, again only in Cartan type $A$. As should come as no surprise given
Theorem~\ref{pti}, 
these two methods, Joseph's method
for computing Goldie ranks in Cartan type $A$
and our method for computing dimensions,
reduce after some book-keeping 
to performing exactly the same computation with Kazhdan-Lusztig
polynomials.
In the last section of the article, we will use this observation to give another
proof of Theorem~\ref{pti},
quite different from
Premet's argument in \cite{Pnew}
which involves reduction
modulo $p$ techniques.

Premet's theorem allows several other classical problems
about $\Prim$ to be attacked using finite $W$-algebra techniques. 
Perhaps our most striking accomplishment along these lines 
is a new proof of M\oe glin's
theorem from \cite{M}, asserting that all completely prime primitive
ideals of $U(\mathfrak{g})$ are induced from one dimensional
representations of parabolic subalgebras. In the rest of the
introduction we will discuss this in more detail
and formulate some other results about
Goldie ranks of primitive quotients in Cartan type $A$ 
obtained using the
link to finite $W$-algebras.
We will also make some 
other apparently new observations about Joseph's Goldie rank
polynomials.
Before we give any more details, we introduce 
some combinatorial language.
\begin{itemize}
\item
A {\em tableau} $A$ 
is a left-justified array of complex numbers with
$\lambda_1$ entries in the bottom row, $\lambda_2$ entries in the next row up, and so on, for some partition 
$\lambda = (\lambda_1 \geq \lambda_2 \geq \cdots)$ of $N$;
we refer to $\lambda$ as the {\em shape} of $A$.
\item 
Two tableaux $A$ and $B$ are {\em row-equivalent}, denoted
$A \sim B$, if one can be obtained from the other by permuting entries within rows.
\item
A tableau is {\em column-strict} if its entries are strictly increasing 
from bottom to top within each column with respect to the partial order $\geq$ on $\C$
defined by $a \geq b$ if $a-b \in \Z_{\geq 0}$.
\item
A tableau is {\em column-connected} if every entry in every row apart from
the bottom row
is one more than the entry immediately below it.
\item
A tableau is {\em column-separated} if it is column-strict
and no two of its columns are linked, where we say that two
columns are {\em linked} if the sets $I$ and $J$ 
of entries from the two columns satisfy 
the following:
\begin{itemize}
\item[$\circ$]
if  $|I| > |J|$ then $i > j > i'$
 for some
$i,i' \in I \setminus J$ and $j \in J \setminus I$;
\item[$\circ$]
if $|I| < |J|$ then $j' < i < j$
 for some
$i \in I \setminus J$ and $j,j' \in J \setminus I$;
\item[$\circ$]
if $|I| = |J|$ then either
$i > j > i' > j'$ or $i' < j' < i < j$
for some $i,i' \in I \setminus J$ and $j,j' \in J \setminus I$.
\end{itemize}
\item
A tableau is {\em standard} if its entries are 
$1,\dots,N$ and they increase
from bottom to top in
each column and from left to right in each row.
\end{itemize}

Now go back to the Lie algebra $\mathfrak{g} = \mathfrak{gl}_N(\C)$.
Let $\mathfrak{t}$ and $\mathfrak{b}$ be the usual choices
of Cartan and Borel subalgebras consisting of diagonal and upper triangular matrices in $\mathfrak{g}$, respectively.
Let $W := S_N$ 
be the Weyl group of $\mathfrak{g}$ with respect to $\mathfrak{t}$, 
identified with the group of all permutation matrices 
in $G := GL_N(\C)$.
Let $\ell$ be the usual length function and
$w_0 \in W$ be the longest element.
Let $\eps_1,\dots,\eps_N \in \mathfrak{t}^*$ be
the dual basis to the basis $x_1,\dots,x_N\in\mathfrak{t}$
given by the diagonal matrix units.
Given any tableau $A$, we attach a weight
$\gamma(A) \in \mathfrak{t}^*$ by letting $a_1,\dots,a_N \in \C$
be the sequence obtained by
reading the entries of $A$ in order down columns starting with the
leftmost column,
then setting 
\begin{equation}\label{newgamma}
\gamma(A) := \sum_{i=1}^N a_i \eps_i.
\end{equation}
Finally let $\Phi^+$ be the positive roots corresponding to $\mathfrak{b}$
and set
\begin{equation}\label{rhodef}
\rho := -\eps_1-2\eps_2-\cdots-N\eps_N,
\end{equation}
which is the usual half-sum of positive roots up to a
convenient normalization.

Given $\alpha \in \mathfrak{t}^*$, let $L(\alpha)$ denote the 
{\em irreducible} $\mathfrak{g}$-module generated by a $\mathfrak{b}$-highest weight vector of weight $\alpha - \rho$.
By Duflo's theorem \cite{Duflo}, the map
\begin{equation*}
I:\mathfrak{t}^* \rightarrow \Prim,
\qquad
\alpha \mapsto I(\alpha) := \ann_{U(\mathfrak{g})} L(\alpha)
\end{equation*}
is surjective. 
In \cite[Th\'eor\`eme 1]{Jduflo} (see also \cite[5.26(1)]{Je}),
Joseph described the fibers of this map explicitly
via the Robinson-Schensted algorithm, as follows.
Take $\alpha \in \mathfrak{t}^*$ and set $a_i := x_i(\alpha)$.
Construct a tableau $Q(\alpha)$ 
by starting from the empty tableau $A_0$, then recursively
inserting the numbers $a_1,\dots,a_N$
into the bottom row using the Schensted insertion algorithm.
So at the $i$th step we are given a tableau $A_{i-1}$
and need to insert $a_i$ into the bottom row of $A_{i-1}$.
If there is no entry $b > a_i$ on this row
then we simply add $a_i$ to the end of the row;
otherwise we replace the leftmost $b > a_i$ on the row with $a_i$,
then repeat the procedure to 
insert $b$ into the next row up.
It is clear from this construction that $Q(\alpha)$ is always row-equivalent to a column-strict tableau.
Now Joseph's fundamental result is that
\begin{equation}\label{josephs}
I(\alpha) = I(\beta)
\quad\Leftrightarrow\quad
Q(\alpha) \sim Q(\beta)
\end{equation}
for any 
$\alpha,\beta \in \mathfrak{t}^*$. 

Thus we have a complete classification of the primitive ideals in 
$U(\mathfrak{g})$.
Our first new result identifies the 
primitive ideals $I$ in this classification that are completely
prime, i.e. the ones for which the quotient $U(\mathfrak{g}) / I$
is a domain.

\begin{Theorem}\label{mt}
For $\alpha \in \mathfrak{t}^*$, the primitive ideal $I(\alpha)$
is completely prime if and only if $Q(\alpha)$ is row-equivalent to a
column-connected tableau.
\end{Theorem}

Of course $I(\alpha)$ is completely prime if and only if
$\grk U(\mathfrak{g}) / I(\alpha) = 1$.
So in view of Theorem~\ref{pti} the completely prime primitive ideals
of $U(\mathfrak{g})$ are related to one dimensional representations of
the finite $W$-algebras $U(\mathfrak{g},e)$.
This is the basic idea for the proof of Theorem~\ref{mt}: we 
deduce it from a classification of
one dimensional representations of $U(\mathfrak{g},e)$
obtained via another result of Premet \cite[Theorem 3.3]{Pabelian}
describing the maximal commutative quotient $U(\mathfrak{g},e)^{\operatorname{ab}}$.

Our next theorem 
constructs
a large family of primitive ideals which are 
{\em induced} in the spirit of \cite[Th\'eor\`eme 8.6]{CB}; again our
proof of this uses finite $W$-algebras in an essential way.

\begin{Theorem}\label{sep}
Suppose we are given 
$\alpha \in \mathfrak{t}^*$
such that $Q(\alpha) \sim A$ for some column-separated tableau $A$. 
Let $\la' = (\la_1' \geq \la_2' \geq \cdots)$ 
be
the transpose of the shape of $A$.
Then we have that
\begin{align*}
I(\alpha) &= \ann_{U(\mathfrak{g})} (U(\mathfrak{g})
\otimes_{U(\mathfrak{p})} F),\\
\grk U(\mathfrak{g}) / I(\alpha) &=
\dim F,
\end{align*}
where $\mathfrak{p}$ is the standard parabolic subalgebra 
with diagonally embedded Levi factor
$\mathfrak{gl}_{\la_1'}(\C) \oplus \mathfrak{gl}_{\la_2'}(\C) \oplus
\cdots$,
and $F$ is the finite dimensional irreducible $\mathfrak{p}$-module
generated by a $\mathfrak{b}$-highest weight vector of weight
$\gamma(A)-\rho$; cf. (\ref{newgamma})--(\ref{rhodef}).
\end{Theorem}

Using these two results we can already recover M\oe glin's
theorem. 

\vspace{2mm}
\noindent
{\bf Corollary}\:(M\oe glin){\bf .}
{\em
Every completely prime primitive ideal $I$ of $U(\mathfrak{g})$
is the annihilator of a module induced from a one dimensional
representation
of a parabolic subalgebra of $\mathfrak{g}$.}

\begin{proof}
Take a completely prime $I \in \Prim$ 
and represent it as $I(\alpha)$ for $\alpha \in \mathfrak{t}^*$.
By Theorem~\ref{mt}, there exists a column-connected tableau
$A \sim Q(\alpha)$.
Since column-connected tableaux are 
obviously column-separated, we then
apply
Theorem~\ref{sep} to deduce that $I = \ann_{U(\mathfrak{g})}
(U(\mathfrak{g}) \otimes_{U(\mathfrak{p})} F)$
for some parabolic $\mathfrak{p}$ and some
$\mathfrak{p}$-module $F$.
Finally observe from its explicit description in Theorem~\ref{sep} that $F$ is actually
one dimensional in the case that $A$ is column-connected.
\end{proof}

We record another piece of folklore peculiar to Cartan type $A$;
it justifies the decision to restrict attention 
for the remainder of the introduction just to weights from
the lattice $P := \bigoplus_{i=1}^N \Z \eps_i$
of {\em integral weights}.
We will give a natural proof of this via finite $W$-algebras, though it
also follows from more classical techniques.

\begin{Theorem}\label{red}
Suppose we are given $\alpha \in \mathfrak{t}^*$ and set $a_i := x_i(\alpha)$.
For fixed $z \in \C$, let
$\mathfrak{g}_z :=
\mathfrak{gl}_{n}(\C)$
where
$n := \#\{i=1,\dots,N\:|\:a_i \in z + \Z\}$,
then set
$\alpha_z := \sum_{j=1}^{n} (a_{i_j}-z)
\eps_j$ where $i_1 < \cdots < i_{n}$
are all the 
$i \in \{1,\dots,N\}$
such that $a_i \in z + \Z$. So $\alpha_z$ is an integral weight for $\mathfrak{g}_z$.
We have that
$$
\grk U(\mathfrak{g}) / I(\alpha)
= \prod_{z}
\grk U(\mathfrak{g}_z) / I(\alpha_z),
$$
where the product is over a set of representatives for the cosets of
$\C$ modulo $\Z$.
\end{Theorem}

In order to say more about Goldie ranks, we need some language related to the geometry of $P$.
A weight $\alpha \in P$ is {\em anti-dominant}
(resp.\ {\em regular anti-dominant}) if it satisfies $x_i(\alpha) \leq
x_{i+1}(\alpha)$
(resp.\ $x_i(\alpha) < x_{i+1}(\alpha)$) 
for each $i=1,\dots,N-1$.
Given any $\alpha \in P$, we let $\delta$ be its
{\em anti-dominant conjugate}, the unique anti-dominant weight in its $W$-orbit,
and then define $d(\alpha)\in W$ to be the unique element of minimal length
such that $\alpha = d(\alpha) \delta$.
Note that stabilizer $W_\delta$ of $\delta$ in $W$
is a parabolic subgroup, and the element $d(\alpha)$ belongs to the set
$D_\delta$ of minimal length $W / W_\delta$-coset
representatives.
For $w \in W$ let
\begin{equation}\label{uc}
\widehat{C}_w := \left\{\alpha \in P\:|\:d(\alpha) = w\right\},
\end{equation}
which is the set of integral weights lying in the 
{\em upper closure} of the chamber containing $w(-\rho)$,
i.e. we have $\alpha 
\in \widehat{C}_w$ if and only if the following 
hold for every $1 \leq i < j \leq N$:
$$
\begin{array}{lcl}
w^{-1}(i) < w^{-1}(j)&\Rightarrow&x_i(\alpha) \leq x_j(\alpha),\\
w^{-1}(i) > w^{-1}(j)&\Rightarrow&x_i(\alpha) > x_j(\alpha).
\end{array}
$$
The upper closures $\widehat{C}_w$ for all $w \in W$ partition the
set $P$ into disjoint subsets.

Recall also the {\em left cells} of $W$, which in the
case of the symmetric group
can be defined in purely combinatorial terms as the equivalence classes
of the relation $\sim_L$ on $W$ defined by
\begin{equation*}
x \sim_L y \Leftrightarrow Q(x) = Q(y).
\end{equation*}
The map $Q$ here comes from the classical Robinson-Schensted bijection 
\begin{equation*}
w \mapsto (P(w), Q(w))
\end{equation*}
from $W$ to the set of all 
pairs of standard tableaux of the same shape
as in e.g. \cite[ch.1]{fulton}; so $P(w)$ is the {\em insertion tableau}
and $Q(w)$ is the {\em recording tableau}.
Comparing with 
our earlier notation, we have that
\begin{equation}\label{rel}
Q(w) = P(w^{-1}) = Q(w(-\rho)),
\end{equation}
hence the connection between left cells in $W$ and
the Duflo-Joseph classification of primitive ideals
from (\ref{josephs}).

We say that $w \in W$ is {\em minimal in its left cell}
if $P(w)$ has the entries $1,\dots,N$ appearing in 
order up columns starting from the leftmost column. 
It is clear from the Robinson-Schensted correspondence that every 
left cell has a unique such minimal representative.
Given any $\alpha \in \widehat{C}_w$, 
the Robinson-Schensted algorithm assembles the tableaux $Q(\alpha)$
and $Q(w(-\rho))=P(w^{-1})$ in exactly the same order, i.e. they
have the same recording tableau $Q(w^{-1}) = P(w)$.  
If $w$ is minimal in its left cell, so this recording tableau
has entries $1,\dots,N$ in 
order up columns,
we therefore have that
\begin{equation}\label{minimal}
\alpha = \gamma(Q(\alpha)))
\end{equation}
for any 
$\alpha\in\widehat{C}_w$ and $w$ that is minimal in its left cell.
This is the reason that the minimal left cell representatives are
particularly
convenient to work with.

At last we can resume the main discussion of Goldie ranks.
In \cite[$\S$5.12]{JgoldieI},
Joseph made the striking discovery that
for each $w \in W$ there is a unique
polynomial $p_w \in \C[\mathfrak{t}^*]$ with the property that
\begin{equation}\label{goldiedef}
\grk U(\mathfrak{g}) / I(\alpha)
= p_w(\delta)
\end{equation}
for each $\alpha \in \widehat{C}_w$, where $\delta$ denotes the
anti-dominant conjugate of $\alpha$.
The $p_w$'s
are Joseph's {\em Goldie rank polynomials}, which have many remarkable
properties. We recall in particular that $p_w$ only depends on the left
cell of $w$. To see this, take any regular anti-dominant $\delta \in P$.
Assuming $w \sim_L w'$ we have that 
$Q(w\delta) = Q(w'\delta)$ 
so $I(w \delta) = I(w' \delta)$ by (\ref{josephs}).
Also $w \delta$ and $w' \delta$ belong to (the interior of) $\widehat{C}_w$
and $\widehat{C}_{w'}$, respectively, by regularity.
Hence (\ref{goldiedef}) gives that
$p_w(\delta) = p_{w'}(\delta)$.
Since the regular anti-dominant weights are Zariski dense 
this implies that $p_w = p_{w'}$ whenever $w \sim_L w'$.

The following theorem, which is 
ultimately deduced from Theorem~\ref{sep}, gives an explicit formula for Goldie
rank polynomials in several important cases, e.g. 
it includes the extreme cases $w = 1$ (when $p_w=1$) and
$w=w_0$ (when it is essentially Weyl's dimension formula), as well as
all situations when the tableau $Q(w)$ has just two rows.

\begin{Theorem}\label{myg}
Suppose we are given 
$w \in W$
such that $Q(w) \sim A$ for some column-separated tableau $A$. 
Then we have that
$$
p_w = 
\prod_{(i,j)} \frac{x_i-x_j}{d(i,j)}
$$
where the product is over all pairs $(i,j)$
of entries from the tableau $A$ such that $i$ is strictly above and in the same column as
$j$, and $d(i,j) > 0$ is the number of rows that $i$ is above $j$.
\end{Theorem}

For general $w$, the polynomials $p_w$ are more complicated
but can be written explicitly in terms of Kazhdan-Lusztig
polynomials. To explain this, and for later use,
we must make one more notational digression.
Recall that the irreducible module $L(\alpha)$ is the unique
irreducible quotient of the {\em Verma module}
$M(\alpha)
:= U(\mathfrak{g}) \otimes_{U(\mathfrak{b})} \C_{\alpha-\rho}$,
where $\C_{\alpha-\rho}$ is the one dimensional $\mathfrak{b}$-module 
of weight $\alpha-\rho$. 
We have the usual {\em decomposition numbers}
$[M(\alpha):L(\beta)] \in \Z_{\geq 0}$ and the {\em inverse decomposition numbers}
$(L(\alpha):M(\beta)) \in \Z$ defined from
\begin{equation}\label{first}
\ch L(\alpha) = \sum_{\beta} 
(L(\alpha):M(\beta)) \ch M(\beta).
\end{equation}
For $w \in W$, we denote $L(w(-\rho))$ and $M(w(-\rho))$ simply by
$L(w)$ and $M(w)$, respectively; in particular, 
$L(w_0)$ is the trivial module.
By the translation principle (see \cite[4.12]{Je}), we have that
\begin{align}\label{kldef}
[M(\alpha):L(\beta)] &= [M(x):L(y)],\\
(L(\alpha):M(\beta)) &= 
\sum_{z \in W_\delta}
(L(x):M(yz)),\label{klform}
\end{align}
for any $\alpha,\beta \in P$ with
the same anti-dominant
conjugate $\delta$, where $x := d(\alpha)$ and $y := d(\beta)$.
Moreover, by the Kazhdan-Lusztig conjecture established
in \cite{BB, BK}, it is known for $x,y\in W$ that
\begin{align}\label{bigkl}
[M(x):L(y)] &= P_{x w_0,y w_0}(1),\\
(L(x):M(y)) &= (-1)^{\ell(x)+\ell(y)} P_{y,x}(1)\label{bigkl2}
\end{align}
where $P_{x,y}(t)$ denotes the Kazhdan-Lusztig polynomial
attached to $x,y \in W$ from \cite{KL}.

The following theorem gives an explicit formula
for the Goldie rank polynomials $p_w$.
It is a straightforward
consequence of Joseph's original approach for computing 
Goldie ranks in Cartan type $A$ from \cite{Jkos},
which we already mentioned in the discussion after Theorem~\ref{pti}.
As was explained to me by Joseph, it can also be deduced from Joseph's general formula for Goldie
rank polynomials (bearing in mind that 
all the scale factors are known in Cartan type $A$).
We give yet another proof in the last section of the article 
via finite $W$-algebras, exploiting Theorem~\ref{pti}.
Recall for the statement that $p_w$ depends only on the left cell of $w$,
so it is sufficient to compute $p_w$ just for the 
minimal left cell representatives.

\begin{Theorem}[Joseph]\label{foldie}
Suppose $w \in W$ is minimal in its left cell.
Let $\la$ be the shape of the tableau $Q(w)$ with transpose
$\la' = (\la_1' \geq \la_2' \geq \cdots)$.
Let $W^\la$ denote
the
parabolic 
subgroup $S_{\la'_1} \times S_{\la_2'}\times\cdots$ of $W=S_N$
and
$D^\la$ be 
the set of maximal length $W^\la \backslash W$-coset representatives.
Then
\begin{equation}\label{bform}
p_w=
\sum_{z \in D^\lambda}
(L(w):M(z)) z^{-1}(h_\la)
\end{equation}
where
$h_\la := 
\!\!\displaystyle\prod_{(i\:j) \in W^\la} \!\frac{x_i - x_j}{j-i}$
(product over all transpositions $(i\:j) \in W^\la$).
\end{Theorem}

Joseph has directed a great deal of attention to the problem of
determining the unknown constants in the Goldie rank polynomials
in Cartan types different from $A$. This led Joseph to conjecture in \cite[Conjecture 
8.4(i)]{Jsur} that Goldie
rank polynomials always take the value $1$ on some integral weight.
Our final result verifies this conjecture in Cartan type $A$.
The proof is a surprisingly simple computation from (\ref{bform}).

\begin{Theorem}\label{one}
Every Goldie rank polynomial takes the value one on some element of
$P$.
More precisely, if $w \in W$ is minimal in its left cell
and $C$ is the unique
tableau of the same shape as $Q(w)$
that has all $1$'s on its bottom row, all $2$'s on the next row up, and so on,
then
$p_w(\alpha) =1$ where
$\alpha := w^{-1} \gamma(C)$.
\end{Theorem}

The remainder of the article is organized as follows.
In $\S$2, we recall the highest weight classification of finite dimensional
irreducible representations of the 
finite $W$-algebra $U(\mathfrak{g},e)$
from \cite[Theorem 7.9]{BKrep}.
Then we compare this with
\cite[Theorem 3.3]{Pabelian}
to determine the highest weights of all the
one dimensional $U(\mathfrak{g},e)$-modules explicitly.
In particular we see from this that every one dimensional
representation of a finite $W$-algebra in
Cartan type $A$ can be 
obtained as the restriction of a one dimensional representation of a
parabolic subalgebra of $\mathfrak{g}$,
a statement which is closely related to M\oe glin's theorem.

Then in $\S$\ref{swhitt} we gather together various 
existing results about 
Whittaker functors
and primitive ideals in Cartan type $A$. 
In fact we need to exploit both sorts of Whittaker functor (invariants
and coinvariants) to deduce our main results.
We point out in particular Remark~\ref{brundans}, in which we 
formulate a
conjecture which would imply a
classification of primitive ideals in $U(\mathfrak{g},e)$ exactly in
the spirit of the Joseph-Duflo classification of $\Prim$.

In $\S$\ref{sm} we use the criterion for
irreducibility of standard modules from \cite[Theorem 8.25]{BKrep}
to establish the first equality in Theorem~\ref{sep}.

In $\S$\ref{sco} we review
the Whittaker coinvariants construction of finite dimensional 
irreducible 
$U(\mathfrak{g},e)$-modules from
\cite[Theorem 8.21]{BKrep}.

In $\S$\ref{sgoldie} we explain the method from \cite[$\S$8.5]{BKrep}
for computing dimensions of finite dimensional
irreducible $U(\mathfrak{g},e)$-modules, and
extract the polynomial on the right hand side of the 
formula (\ref{bform}) from this.

Finally we explain the alternative proof of Theorem~\ref{pti}
and derive all the other new results formulated in this introduction
in $\S$\ref{sproofs}.

\vspace{2mm}
\noindent
{\em Acknowledgements.}
My interest in reproving M\oe glin's theorem using finite $W$-algebras
was sparked in the first place by a conversation with Alexander Premet and Anthony Joseph
at the Oberwolfach meeting on ``Enveloping Algebras'' in March 2005.
I would like to thank Alexander Premet for some inspiring discussions
and encouragement since then, most recently at the
``Representation Theory of Algebraic Groups and Quantum Groups'' conference
in Nagoya in August, 2010 where I learnt about the new results in
\cite{Pnew}.
I also thank Anthony Joseph for his helpful comments on the first
draft of the article.

\section{One dimensional  representations}\label{s1d}

In this section we recall some basic facts about the representation 
theory of finite $W$-algebras in Cartan type
$A$ from \cite{BKrep}, and then deduce a classification of one
dimensional
representations of these algebras.
We continue with the basic Lie theoretic notation from the introduction,
in particular, $\mathfrak{g} = \mathfrak{gl}_N(\C)$
and $\mathfrak{t}$ and $\mathfrak{b}$ are the usual choices of
Cartan and Borel subalgebra.

Let $\lambda = (p_n \geq \cdots \geq p_1)$ be a fixed partition of $N$.
For each $i=1,\dots,n-1$, pick non-negative integers
$s_{i,i+1}$ and $s_{i+1,i}$ such that
$s_{i,i+1}+s_{i+1,i} = p_{i+1}-p_i$. Then 
set
$s_{i,j} := s_{i,i+1}+s_{i+1,i+2}+\cdots+s_{j-1,j}$ 
and 
$s_{j,i} := s_{j,j-1}+\cdots+s_{i+2,i+1}+s_{i+1,i}$
for $1 \leq i \leq j \leq n$.
This defines a {\em shift matrix} 
$\sigma = (s_{i,j})_{1 \leq i,j \leq n}$ 
in the sense of \cite[(2.1)]{BKshifted}.
Let $l := p_n$ for short, which is called the {\em level} in \cite{BKshifted}.

We visualize this data by means of a {\em pyramid} $\pi$ 
of boxes drawn in an $n \times l$ rectangle, so that there is a box in row $i$ and column $j$
for each $1 \leq i \leq n$ and $1+s_{n,i} \leq j \leq l - s_{i,n}$
(where rows and columns are indexed as in a matrix).
Note that there are $p_i$ boxes in the $i$th row for each $i=1,\dots,n$.
Let $q_j$ be the number of boxes in the $j$th column
for $j=1,\dots,l$.
Also number the boxes of $\pi$ by $1,\dots,N$ working in order down columns starting from the leftmost column, and
 write $\row(k)$ and $\col(k)$ for the row and column numbers of the $k$th box.
For example, for $\lambda = (3,2,1)$ there are four possible
choices for $\sigma$ with corresponding pyramids
$$
\sigma = \left(\begin{array}{lll}0&1&2\\0&0&1\\0&0&0\end{array}\right)
\leftrightarrow
\:\pi=
\begin{picture}(39,0)
\put(3,-16){\line(0,1){36}}
\put(15,-16){\line(0,1){36}}
\put(27,-16){\line(0,1){24}}
\put(39,-16){\line(0,1){12}}
\put(3,-16){\line(1,0){36}}
\put(3,-4){\line(1,0){36}}
\put(3,8){\line(1,0){24}}
\put(3,20){\line(1,0){12}}
\put(9,14){\makebox(0,0){$1$}}
\put(9,2){\makebox(0,0){$2$}}
\put(9,-10){\makebox(0,0){$3$}}
\put(21,2){\makebox(0,0){$4$}}
\put(21,-10){\makebox(0,0){$5$}}
\put(33,-10){\makebox(0,0){$6$}}
\end{picture}\:,
\qquad
\sigma = \left(\begin{array}{lll}0&1&1\\0&0&0\\1&1&0\end{array}\right)
\leftrightarrow
\:\pi=
\begin{picture}(39,0)
\put(3,-16){\line(0,1){12}}
\put(15,-16){\line(0,1){36}}
\put(27,-16){\line(0,1){36}}
\put(39,-16){\line(0,1){24}}
\put(3,-16){\line(1,0){36}}
\put(3,-4){\line(1,0){36}}
\put(15,8){\line(1,0){24}}
\put(27,20){\line(-1,0){12}}
\put(9,-10){\makebox(0,0){$1$}}
\put(21,14){\makebox(0,0){$2$}}
\put(21,2){\makebox(0,0){$3$}}
\put(21,-10){\makebox(0,0){$4$}}
\put(33,-10){\makebox(0,0){$6$}}
\put(33,2){\makebox(0,0){$5$}}
\end{picture}\:,
$$
$$
\sigma = \left(\begin{array}{lll}0&0&1\\1&0&1\\1&0&0\end{array}\right)
\leftrightarrow
\:\pi=\begin{picture}(39,0)
\put(3,-16){\line(0,1){24}}
\put(15,-16){\line(0,1){36}}
\put(27,-16){\line(0,1){36}}
\put(39,-16){\line(0,1){12}}
\put(3,-16){\line(1,0){36}}
\put(3,-4){\line(1,0){36}}
\put(3,8){\line(1,0){24}}
\put(15,20){\line(1,0){12}}
\put(9,2){\makebox(0,0){$1$}}
\put(9,-10){\makebox(0,0){$2$}}
\put(21,14){\makebox(0,0){$3$}}
\put(21,2){\makebox(0,0){$4$}}
\put(21,-10){\makebox(0,0){$5$}}
\put(33,-10){\makebox(0,0){$6$}}
\end{picture}\:,
\qquad
\sigma = \left(\begin{array}{lll}0&0&0\\1&0&0\\2&1&0\end{array}\right)
\leftrightarrow
\:\pi=\begin{picture}(39,0)
\put(3,-16){\line(0,1){12}}
\put(15,-16){\line(0,1){24}}
\put(27,-16){\line(0,1){36}}
\put(39,-16){\line(0,1){36}}
\put(3,-16){\line(1,0){36}}
\put(3,-4){\line(1,0){36}}
\put(15,8){\line(1,0){24}}
\put(27,20){\line(1,0){12}}
\put(9,-10){\makebox(0,0){$1$}}
\put(21,2){\makebox(0,0){$2$}}
\put(21,-10){\makebox(0,0){$3$}}
\put(33,14){\makebox(0,0){$4$}}
\put(33,2){\makebox(0,0){$5$}}
\put(33,-10){\makebox(0,0){$6$}}
\end{picture}\:.
$$
If $\sigma$ is upper-triangular then $\pi$
coincides with
the usual Young diagram of the partition $\lambda$; we refer to this as
the {\em left-justified} case.

By a {\em $\pi$-tableau}, we mean a filling 
of the boxes of the pyramid $\pi$
by complex numbers; the left-justified tableaux 
from the introduction are a special case.
The definitions of {\em column-strict},
{\em column-connected}
and {\em row-equivalence} formulated in 
the introduction in the left-justified case
extend without change to $\pi$-tableaux.
Also say a $\pi$-tableau $A$ is {\em row-standard}
if its entries are non-decreasing along rows from left to right, meaning that
$a \not> b$ whenever $a$ and $b$ are two entries from the same row with
$a$ located to the left of $b$.

We next define two essential maps from $\pi$-tableaux to $\mathfrak{t}^*$,
denoted $\gamma$ and $\rho$ and
called {\em column reading} and {\em row reading}, respectively.
First, for a $\pi$-tableau $A$, we
let \begin{equation}\label{gammadef}
\gamma(A) := \sum_{i=1}^n a_i \eps_i
\end{equation} 
where 
$(a_1,\dots,a_N)$ is the sequence of complex numbers obtained by reading the entries of $A$
 in order down columns starting with the leftmost column;
so $a_i$ is the entry in the $i$th box of $A$.
For $\rho(A)$, we first need to convert $A$ into a row-standard
$\pi$-tableau, which we do by repeatedly transposing pairs of entries
$a > b$ in the same row with $a$ located to the left of $b$ until 
we get to a (uniquely determined) row-standard tableau $A'$.
Then let 
\begin{equation}\label{rhoAdef}
\rho(A) := \sum_{i=1}^n a_i' \eps_i
\end{equation}
where
$(a_1',\dots,a_n')$ is the sequence obtained by reading the entries
of $A'$ in order along rows starting with the top row.
Note the map $\gamma$ is obviously bijective, but $\rho$ is definitely
not.

Let $e \in \mathfrak{g}$ be the
nilpotent matrix
$$
e := \sum_{\substack{1 \leq i,j \leq N\\
\row(i) = \row(j)\\
\col(i) =\col(j)-1}} e_{i,j}
$$
 of Jordan type $\lambda$.
Here $e_{i,j}$ denotes the $ij$-matrix unit.
Introduce a $\Z$-grading $\mathfrak{g} = \bigoplus_{d \in \Z} \mathfrak{g}(d)$
by declaring that $e_{i,j}$ is of degree $2(\col(j)-\col(i))$;
in particular, $e$ is homogeneous of degree $2$. 
Let $\mathfrak{m} := \bigoplus_{d < 0} \mathfrak{g}(d)$,
$\mathfrak{h} := \mathfrak{g}(0)$ and $\mathfrak{p} := \bigoplus_{d \geq 0} 
\mathfrak{g}(d)$.
So $\mathfrak{p}$ is the standard parabolic subalgebra 
with Levi factor $\mathfrak{h}$, and
$\mathfrak{h}$ is just the diagonally embedded subalgebra
$\mathfrak{gl}_{q_1}(\C)\oplus\cdots\oplus \mathfrak{gl}_{q_l}(\C)$.
Let $\mathfrak{g}^e$ (resp.\ $\mathfrak{t}^e$) be the centralizer of $e$
in $\mathfrak{g}$ (resp.\ $\mathfrak{t})$.
It is important that
$\mathfrak{g}^e \subseteq \mathfrak{p}$.

Let
$\chi:\mathfrak{m} \rightarrow \C$ be the Lie algebra homomorphism
$x \mapsto (x,e)$ where $(.,.)$ is the trace form.
Let $\mathfrak{m}_\chi := \{x - \chi(x)\:|\:x \in \mathfrak{m}\}
\subseteq U(\mathfrak{m})$.
The {\em finite $W$-algebra} is the following subalgebra of $U(\mathfrak{p})$:
\begin{equation}\label{fw}
U(\mathfrak{g},e) := \{u \in U(\mathfrak{p})\:|\:
\mathfrak{m}_\chi u \subseteq 
U(\mathfrak{g}) \mathfrak{m}_\chi\}.
\end{equation}
This definition originates in work of 
Kostant \cite{K}, Lynch \cite{Ly} and M\oe glin \cite{MW}, 
and is a special
case of the construction due to Premet \cite{Pslice} and then
Gan and Ginzburg \cite{GG} of non-commutative filtered
deformations of the coordinate algebra of the Slodowy slice associated
to the nilpotent orbit $G \cdot e$; 
the terminology ``finite $W$-algebra'' has emerged because
they are the finite dimensional analogues of the vertex $W$-algebras
constructed in \cite{KRW}.
Of course the definition 
depends implicitly on the choice of grading (hence on $\pi$), but up to isomorphism the algebra
$U(\mathfrak{g},e)$ is independent of this choice; see 
\cite[Corollary 10.3]{BKshifted}.
More conceptual proofs of this independence (valid in all Cartan types) were given subsequently in \cite[Theorem 1]{BG} and \cite[Proposition 3.1.2]{Lsymplectic}.

A special feature of the Cartan type $A$ case is that a complete set of 
generators and relations for $U(\mathfrak{g},e)$ is known; 
see \cite[Theorem 10.1]{BKshifted}. The generators are certain explicit elements 
\begin{align*}
\{D_i^{(r)}\:&|\:1 \leq i \leq n, r > 0\}\\
\{E_i^{(r)}\:&|\:1 \leq i < n, r > s_{i,i+1}\}\\
\{F_i^{(r)}\:&|\:1 \leq i < n, r > s_{i+1,i}\}
\end{align*}
of $U(\mathfrak{p})$ defined in \cite[$\S$9]{BKshifted},
and the relations are the defining relations for the
shifted Yangian $Y_n(\sigma)$ 
recorded in \cite[(2.4)--(2.15)]{BKshifted},
together with the relations $D_1^{(r)} = 0$ for $r > p_1$.
These generators and relations were exploited in \cite{BKrep} to classify
the
finite dimensional irreducible $U(\mathfrak{g},e)$-modules.

To recall this classification in more detail, by a {\em highest weight vector} 
in a $U(\mathfrak{g},e)$-module, we mean
a common eigenvector for all $D_i^{(r)}$
which is annihilated by all $E_j^{(s)}$.
Assume that $v_+$ is a non-zero highest weight vector in a left module.
Let $a_i^{(r)}\in\C$ be defined from
$D_i^{(r)} v_+ = a_i^{(r)} v_+$ and 
define $a_{i,1},\dots,a_{i,p_i} \in \C$ by factoring
\begin{equation}\label{factor}
u^{p_i} + a_i^{(1)} u^{p_i-1}+\cdots+a_i^{(p_i)} =
(u+a_{i,1})\cdots(u+a_{i,p_i}).
\end{equation}
Combining \cite[Theorem 3.5]{BKrep} for $j=i$ with the definition \cite[(2.34)]{BKrep},
it follows that the elements $D_i^{(r)}$ for $r > p_i$ 
lie in the left ideal of $U(\mathfrak{g},e)$ generated by all $E_j^{(s)}$, hence
$a_i^{(r)} = 0$ for $r > p_i$.
So we have for all $r > 0$ that
\begin{equation}\label{esf}
D_i^{(r)} v_+ = e_r(a_{i,1},\dots,a_{i,p_i}) v_+,
\end{equation}
where $e_r(a_{i,1},\dots,a_{i,p_i})$ is the $r$th elementary symmetric polynomial in the complex numbers
$a_{i,1},\dots,a_{i,p_i}$.
We record this by writing the complex numbers
$a_{i,1}-i,\dots,a_{i,p_i}-i$ into the boxes on the $i$th row of the pyramid $\pi$ 
to obtain a $\pi$-tableau $A$, which we refer to as the {\em type} of the original highest weight vector $v_+$.
Of course $A$ here is defined only up to row-equivalence.

Conversely, given a $\pi$-tableau $A$, 
there is a unique (up to isomorphism) irreducible left $U(\mathfrak{g},e)$-module
$L(A,e)$ generated by a highest weight vector of type $A$,
with $L(A,e) \cong L(B,e)$ if and only if
$A \sim B$.
The module $L(A,e)$ is constructed in \cite[$\S$6.1]{BKrep} as the unique irreducible quotient of the
{\em Verma module} $M(A,e)$, which is the universal highest weight module of type $A$; see also \cite[$\S$4.2]{BGK} for a 
different construction of Verma modules which avoids the explicit use
of generators and relations (so makes sense in other Cartan types).

\begin{Remark}\rm
A basic question is to compute the composition multiplicities
$[M(A,e):L(B,e)]$.
In \cite[Conjecture 7.17]{BKrep}, we conjectured
for any $\pi$-tableaux $A$ and $B$ with integer entries that
\begin{equation}\label{conj}
[M(A,e):L(B,e)] = [M(\rho(A)):L(\rho(B))],
\end{equation}
the numbers on the right hand side being known by 
(\ref{kldef}) and (\ref{bigkl}).
Although not needed in the present article,
we want to point out that this conjecture is now a theorem of Losev; see
\cite[Theorems 4.1 and 4.3]{LcatO}.
Strictly speaking, 
to get from Losev's result to (\ref{conj})
one needs to identify the Verma modules 
$M(A,e)$ defined here with the ones in \cite{LcatO},
but this has now been checked thanks to some recent work of Brown and
Goodwin \cite{BrG};
see the proof of Theorem~\ref{labels} below for a fuller discussion.
In arbitrary standard Levi type, there is an analogous
conjecture formulated roughly in \cite{VD}, which can also be proved
using Losev's work.
\end{Remark}

The highest weight classification of finite dimensional irreducible
$U(\mathfrak{g},e)$-modules is as follows.

\begin{Theorem}[{\cite[Theorem 7.9]{BKrep}}]\label{fdc}
For a $\pi$-tableau $A$, 
$L(A,e)$ is finite dimensional if and only if $A$
is row-equivalent to a column-strict tableau.
Hence, 
as $A$ runs over a set of representatives for the row-equivalence classes of column-strict $\pi$-tableaux,
the modules
$\{L(A,e)\}$ give a complete set of pairwise inequivalent finite dimensional irreducible left $U(\mathfrak{g},e)$-modules.
\end{Theorem}

The proof of the ``if'' part of Theorem~\ref{fdc}
given in \cite{BKrep}
is quite straightforward, and is based on the construction of 
another family of $U(\mathfrak{g},e)$-modules
called {standard modules}
indexed by column-strict tableaux.
To define these, recall the weight $\rho$ from (\ref{rhodef}),
and also 
introduce the special weight
\begin{equation}\label{betadef}
\beta :=\!\!\!\!\!\! \sum_{\substack{1 \leq i,j \leq N \\ \col(i) > \col(j)}}
(\eps_i - \eps_j) = \!
\sum_{i=1}^N ((q_1+\cdots+q_{\col(i)-1}) - (q_{\col(i)+1} + \cdots + q_l))\eps_i
\in \mathfrak{t}^*.
\end{equation}
This is the same as the weight $\beta$ 
defined in \cite{BGK}, which is
important because of \cite[Corollary 2.9]{BGK} 
(reproduced in Theorem~\ref{twist}
below).
Notice 
that $A$ is column-strict if and only if
$\gamma(A) - \beta - \rho$
is a dominant weight for the Lie algebra $\mathfrak{h} = \mathfrak{g}(0)$
with respect to the Borel subalgebra $\mathfrak{b}\cap\mathfrak{h}$.
Assuming that is the case,
there is a finite dimensional irreducible $\mathfrak{p}$-module 
$V(A)$ generated by a $\mathfrak{b}$-highest weight vector 
of this weight.
Then we restrict the left $U(\mathfrak{p})$-module $V(A)$ to
the subalgebra $U(\mathfrak{g},e)$
to obtain the {\em standard module} denote $V(A,e)$.
Thus $V(A,e) = V(A)$ as vector spaces, but we use different notation 
since 
one is a $U(\mathfrak{g},e)$-module and the other is a $U(\mathfrak{p})$-module.
As observed in the last paragraph of the proof of \cite[Theorem 7.9]{BKrep},
the original $\mathfrak{b}$-highest weight vector in $V(A)$
is a highest weight vector of type $A$ in $V(A,e)$;
this can also be checked directly 
by arguing as in the proof of 
\cite[Lemma 5.4]{BGK}.
It follows that $L(A,e)$ is a composition factor of the finite dimensional module 
$V(A,e)$, hence  $L(A,e)$ is indeed finite dimensional when $A$ is column-strict.

We are interested next in one dimensional modules.
It is obvious from the definitions that $V(A)$ 
is one dimensional if and only if
$A$ is column-connected. Since $L(A,e)$ is a subquotient of $V(A,e)$,
it follows that $L(A,e)$ is one-dimensional if $A$ is
row-equivalent to a column-connected tableau.
We are going to prove 
the converse of this statement to obtain the following classification of
one dimensional $U(\mathfrak{g},e)$-modules. 
The possibility of doing this was suggested already by Losev in the discussion in the paragraph after \cite[Theorem 5.2.1]{L1D}.

\begin{Theorem}\label{class}
For a $\pi$-tableau $A$, $L(A,e)$ is one dimensional
if and only if $A$ is row-equivalent to a column-connected tableau.
Hence, as $A$ runs over a set of representatives for the row-equivalence classes
of column-connected $\pi$-tableaux, the modules $\{L(A,e)\}$
give a complete set of pairwise inequivalent one dimensional left $U(\mathfrak{g},e)$-modules.
\end{Theorem}

\begin{Corollary}\label{mc}
Every one dimensional left $U(\mathfrak{g},e)$-module is
isomorphic to a standard module $V(A,e)$ for some
column-connected $\pi$-tableau $A$, so arises as the restriction of
a one dimensional $U(\mathfrak{p})$-module.
\end{Corollary}

The rest of the section is devoted to proving Theorem~\ref{class} and
its corollary.
To do this,
we need to review the following theorem of Premet describing the algebra
 $U(\mathfrak{g},e)^{\ab}$, that is,
the quotient of $U(\mathfrak{g},e)$
by the two-sided ideal generated by all commutators $[x,y]$ for $x,y \in U(\mathfrak{g},e)$. 
Of course 
one dimensional $U(\mathfrak{g},e)$-modules are identified with
one dimensional $U(\mathfrak{g},e)^{\ab}$-modules.
It is convenient at this point to set $p_0 := 0$.

\begin{Theorem}[{\cite[Theorem 3.3]{Pabelian}}]\label{pt}
The algebra $U(\mathfrak{g},e)^{\ab}$ is a free polynomial algebra of
rank $l$ generated by the images of the elements
\begin{equation}\label{elts}
\{D_i^{(r)}\:|\:1 \leq i \leq n, 1 \leq r \leq p_{i} - p_{i-1}\}.
\end{equation}
\end{Theorem}

Premet's proof of Theorem~\ref{pt} is in two parts. The first step is to show
that $U(\mathfrak{g},e)^{\ab}$ is generated by the images of the commuting
elements listed in (\ref{elts}). This is 
a straightforward consequence of the defining relations for $U(\mathfrak{g},e)$
from \cite{BKshifted}, and is explained in the first two paragraphs of the proof of
\cite[Theorem 3.3]{Pabelian}.
Thus,
letting $X \cong \mathbb{A}^l$ be the affine space with algebraically independent
coordinate
functions
$\{T_i^{(r)}\:|\:1 \leq i \leq n, 1 \leq r \leq p_i - p_{i-1}\}$,
 there is a surjective map 
\begin{equation}\label{mor}
\C[X] 
 \twoheadrightarrow
U(\mathfrak{g},e)^{\ab},
\qquad T_i^{(r)} \mapsto D_i^{(r)}.
\end{equation}
This map 
identifies $\operatorname{Specm} U(\mathfrak{g},e)^{\ab}$ with a closed
subvariety of $X$.
Then to complete the proof Premet shows quite indirectly
that $\dim \operatorname{Specm}
U(\mathfrak{g},e)^{\ab} \geq l$, hence
$\operatorname{Specm} U(\mathfrak{g},e)^{\ab}
=X$ and the surjective map is an isomorphism.
In the next paragraph, we will explain an alternative argument
for this second step using the following elementary lemma.

\begin{Lemma}\label{Stup}
Given complex numbers $a_i^{(r)}$ for $1 \leq i \leq n$ and $1 \leq r \leq p_i - p_{i-1}$, 
there are complex numbers $a_{i,j}$ for $1 \leq i \leq n$ and $1 \leq j \leq p_i$ such that
\begin{align}\label{id1}
a_{i,p_i-p_{i-1}+r} &= a_{i-1,r}&&\text{for }1 \leq r \leq p_{i-1},\\
e_r(a_{i,1},\dots,a_{i,p_i}) &= a_{i}^{(r)}&&\text{for }1 \leq r \leq p_i-p_{i-1}.\label{id2}
\end{align}
\end{Lemma}

\begin{proof}
We prove existence of numbers $a_{i,j}$ for $1 \leq j \leq p_i$
satisfying (\ref{id1})--(\ref{id2})
by induction on $i=1,\dots,n$.
For the base case $i=1$, we define $a_{1,1},\dots,a_{1,p_1}$ from the
factorization (\ref{factor}), and (\ref{id1})--(\ref{id2}) are clear.
For the induction step, suppose we have already found $a_{i-1,1},\dots,a_{i-1,p_{i-1}}$. 
Define $a_{i,p_i-p_{i-1}+1},\dots,a_{i,p_i}$ 
so that (\ref{id1}) holds. Then we need to find complex numbers
$a_{i,1},\dots,a_{i,p_i-p_{i-1}}$ satisfying (\ref{id2}).
The equations (\ref{id2}) are equivalent to the equations
$$
b_i^{(r)} = 
a_i^{(r)} - \sum_{s=0}^{r-1} b_i^{(s)} e_{r-s}(a_{i-1,1},\dots,a_{i-1,p_{i-1}})
$$
for $1 \leq r \leq p_i -p_{i-1}$,
where $b_i^{(r)}$ denotes $e_r(a_{i,1},\dots,a_{i,p_i-p_{i-1}})$,
Proceeding by induction on $r=1,\dots,p_i-p_{i-1}$, we solve these equations
uniquely for $b_i^{(r)}$ and then define $a_{i,1},\dots,a_{i,p_i-p_{i-1}}$
by factoring
$$
u^{p_i-p_{i-1}} + b_i^{(1)} u^{p_i-p_{i-1}-1}+\cdots+b_i^{(p_i-p_{i-1})}
= (u+a_{i,1}) \cdots (u+a_{i,p_i-p_{i-1}}).
$$
This does the job.
\end{proof}

Now take any point $x \in X$,
set $a_i^{(r)} := T_i^{(r)}(x)$,
and then 
define $a_{i,j}$ according to Lemma~\ref{Stup}.
Because of (\ref{id1}), there is a column-connected $\pi$-tableau $A$ having entries
$a_{i,1}-i,\dots,a_{i,p_i}-i$ in its $i$th row for each $i=1,\dots,n$.
This tableau $A$ is unique up to row-equivalence, indeed, any two choices 
for $A$ agree up to reordering columns of the same height.
As we have already observed,
the assumption that $A$ is column-connected means that the
standard module $V(A,e)$ is one dimensional, hence 
so is $L(A,e) \cong V(A,e)$. By (\ref{esf}) and (\ref{id2}),
we see that $D_i^{(r)}$ acts on $L(A,e)$ by the scalar $a_i^{(r)}$,
showing that the point $x$ lies in $\operatorname{Specm} U(\mathfrak{g},e)^{\operatorname{ab}}$.
Thus we have established that
$\operatorname{Specm} U(\mathfrak{g},e)^{\ab}
=X$, so the map (\ref{mor}) is indeed an isomorphism
as required for the alternative proof of 
the second part of Theorem~\ref{pt} promised above.

This argument shows moreover
that every one dimensional left $U(\mathfrak{g},e)$-module is isomorphic to $L(A,e) \cong V(A,e)$ 
for some column-connected $\pi$-tableau $A$, which is enough to
complete the proofs of
Theorem~\ref{class} and Corollary~\ref{mc}.

\section{Whittaker functors and Duflo-Joseph classification}\label{swhitt}

In this section we review the definitions of the two sorts of 
Whittaker functors and explain some of
the results of Premet and Losev linking 
finite dimensional $U(\mathfrak{g},e)$-modules
to $\Prim$.

For any associative algebra $A$, we denote the category of
all left (resp.\ right) $A$-modules
by $A\lmod$ (resp.\ $\rmod A$).
If $M$ is a left $U(\mathfrak{g})$-module,
it is clear from (\ref{fw}) that the space
$H^0(\mathfrak{m}_\chi, M) := \{v \in M \:|\:\mathfrak{m}_\chi v = \bz\}$
of {\em Whittaker invariants}
is stable under left multiplication by elements
of $U(\mathfrak{g},e)$, hence it is a left
$U(\mathfrak{g},e)$-module.
So we get the functor
\begin{equation}\label{inv}
H^0(\mathfrak{m}_\chi, ?):U(\mathfrak{g})\lmod \rightarrow
U(\mathfrak{g},e)\lmod.
\end{equation}
Instead suppose that $M$ is a right $U(\mathfrak{g})$-module.
Then, by (\ref{fw}) again, the space 
$H_0(\mathfrak{m}_\chi, M) := M / M \mathfrak{m}_\chi$
of {\em Whittaker coinvariants}
is naturally a right $U(\mathfrak{g},e)$-module. 
So we have the functor
\begin{equation}\label{coinv}
H_0(\mathfrak{m}_\chi, ?):\rmod U(\mathfrak{g}) \rightarrow
\rmod U(\mathfrak{g},e).
\end{equation}
In the remainder of the section 
we review some of the basic properties of these two Whittaker functors.
Although not used here, we remark that one can also combine these functors
to obtain a remarkable
functor $H^0_0(\mathfrak{m}_\chi, ?)$ on
bimodules introduced originally by Ginzburg;
see
 \cite[$\S$3.3]{Gin} and \cite[$\S$3.5]{Lclass}.

We begin with the functor $H^0(\mathfrak{m}_\chi,?)$.
Let 
$(U(\mathfrak{g}),\mathfrak{m}_\chi)\lmod$
be the full subcategory of
$U(\mathfrak{g})\lmod$ consisting of all 
modules on which $\mathfrak{m}_\chi$ acts locally nilpotently.
By Skryabin's theorem \cite{Skryabin} (see also \cite[$\S$6]{GG}),
the functor $H^0(\mathfrak{m}_\chi,?)$
restricts to an equivalence of categories
\begin{equation*}
H^0(\mathfrak{m}_\chi, ?):(U(\mathfrak{g}),\mathfrak{m}_\chi)\lmod \rightarrow
U(\mathfrak{g},e)\lmod.
\end{equation*}
The quasi-inverse equivalence is the {\em Skryabin functor}
\begin{equation}\label{skryabinf}
S_\chi
: U(\mathfrak{g},e)\lmod
\rightarrow
(U(\mathfrak{g}),\mathfrak{m}_\chi)\lmod
\end{equation}
defined by tensoring with the $(U(\mathfrak{g}), U(\mathfrak{g},e))$-bimodule
$U(\mathfrak{g}) / U(\mathfrak{g}) \mathfrak{m}_\chi$.

This equivalence has proved useful for the study of primitive ideals
in $U(\mathfrak{g})$.
For a two-sided ideal $I$ of $U(\mathfrak{g})$, we define its
associated variety $\VA(I)$ as in \cite[$\S$9.3]{Ja},
viewing it as a closed subvariety of $\mathfrak{g}$ via the trace form.
Let $\VA'(I)$ denote the image of $\VA(I)$ under the natural 
projection $\mathfrak{g} \twoheadrightarrow [\mathfrak{g},\mathfrak{g}]
= \mathfrak{sl}_N(\C)$.
By Joseph's irreducibility theorem, it is known that
$\VA'(I)$ is the closure of a single 
nilpotent orbit for every $I \in \Prim$.
This follows in Cartan type $A$ from \cite[$\S$3.3]{JI}; for other
Cartan types see
\cite[$\S$3.10]{Joseph} as well as
\cite[Corollary 4.7]{Vogan} and \cite[Remark 3.4.4]{Lclass} 
for alternative proofs (the second of which goes via finite $W$-algebras
in the spirit of the present article).
Let $\lPrim$ denote the set of $I \in \Prim$
such that $\VA'(I)$ is the closure of the orbit $G \cdot e$
of all nilpotent 
matrices of Jordan type $\lambda$.

Given any non-zero left $U(\mathfrak{g},e)$-module $L$, we get
a two-sided ideal 
\begin{equation}\label{il}
I(L) := 
\ann_{U(\mathfrak{g})}
S_\chi(L)
\end{equation}
of $U(\mathfrak{g})$ by applying Skryabin's functor (\ref{skryabinf})
and then taking the annihilator.
If $L$ is irreducible then
Skryabin's theorem implies that $I(L) \in \Prim$.
The following fundamental theorem of Premet
implies moreover that $I(L) \in \lPrim$ 
if $L$ is 
finite dimensional and irreducible.
Premet's proof of this result also uses
Joseph's irreducibility theorem.
Although not needed here, we remark  that the converse 
of the second statement of the theorem is also true by 
\cite[Theorem 1.2.2(ii),(ix)]{Lsymplectic}.

\begin{Theorem}[{\cite[Theorem 3.1]{Pslodowy}}]\label{t31}
For any non-zero left $U(\mathfrak{g},e)$-module $L$
we have that $\VA'(I(L)) \supseteq \overline{G \cdot e}$.
Moreover, if
$L$ is finite dimensional then $\VA'(I(L)) = \overline{G \cdot e}$.
\end{Theorem}

Recalling Theorem~\ref{fdc},
this gives us an ideal $I(L(A,e)) \in \lPrim$
for each column-strict $\pi$-tableau $A$.
The next theorem explains how to identify this primitive ideal in 
the Duflo labelling from the introduction.
It is a special case of a general result of Losev \cite[Theorem 5.1.1]{L1D}
(a closely related statement was conjectured in \cite[$\S$5.1]{BGK}).

\begin{Theorem}\label{labels}
For any $\pi$-tableau $A$, we have that
$$
I(L(A,e)) = I(\rho(A)),
$$
where $\rho(A) \in \mathfrak{t}^*$ is defined by (\ref{rhoAdef}).
\end{Theorem}

\begin{proof}
Recall we have labelled the boxes of $\pi$ in order down columns starting from the leftmost column. 
Let $1',2',\dots,N'$ be the sequence of integers obtained by reading
these labels 
from left to right along rows starting from the top row.
There is a unique permutation $w \in W$ such that
$w(i) = i'$ for each $i=1,\dots,N$.
Let $\mathfrak{b}' := w\cdot\mathfrak{b} = \langle e_{i',j'}\:|\:1 \leq i \leq j \leq N\rangle$
and $\rho' := w\rho = -\sum_{i=1}^N i \eps_{i'}$.
For any $\alpha' \in \mathfrak{t}^*$, 
let $L'(\alpha')$ be the irreducible $\mathfrak{g}$-module generated
by a $\mathfrak{b}'$-highest weight vector of weight $\alpha' - \rho'$.
Now take a $\pi$-tableau $A$ and
let $\rho'(A) := w \rho(A)$.
An easy argument involving twisting the action by $w$ 
shows that
$\ann_{U(\mathfrak{g})} L'(\rho'(A)) = \ann_{U(\mathfrak{g})} L(\rho(A))
\stackrel{\text{def}}{=} I(\rho(\alpha))$.
Thus to complete the proof of the theorem it suffices to show that
\begin{equation}\label{prob}
I(L(A,e)) \stackrel{\text{def}}{=} \ann_{U(\mathfrak{g})} S_\chi(L(A,e))
= 
\ann_{U(\mathfrak{g})} L'(\rho'(A)).
\end{equation}
We will ultimately deduce this from \cite[Theorem 5.1.1]{L1D},
which is in phrased in terms of 
\cite{BGK} highest weight theory.

To recall a little of this theory,
for $\mathfrak{a} \in \{\mathfrak{g}, \mathfrak{p},
\mathfrak{h}, \mathfrak{m}, \mathfrak{b}, \mathfrak{b}'\}$,
let $\mathfrak{a}_0$
be the zero weight space of $\mathfrak{a}$ 
for the adjoint action of the torus $\mathfrak{t}^e$.
In particular, we have that
$\mathfrak{g}_0 = \langle e_{i,j}\:|\:\row(i) = \row(j)\rangle 
\cong \mathfrak{gl}_{p_1}(\C)
\oplus\cdots\oplus \mathfrak{gl}_{p_n}(\C)$,
while $\mathfrak{p}_0 = \mathfrak{b}_0 = \mathfrak{b}'_0$ and $\mathfrak{h}_0 = \mathfrak{t}$.
We have in front of us the necessary data to define 
another finite $W$-algebra
$U(\mathfrak{g}_0, e) \subseteq U(\mathfrak{p}_0)$,
which plays the role of ``Cartan subalgebra.''
Choose a parabolic subalgebra $\mathfrak{q}$ of $\mathfrak{g}$
with Levi factor $\mathfrak{g}_0$ by setting 
$\mathfrak{q} := \mathfrak{g}_0 + \mathfrak{b}'
= \langle e_{i,j}\:|\:\row(i) \leq \row(j)\rangle$.
This choice determines a certain
$(U(\mathfrak{g},e),U(\mathfrak{g}_0,e))$-bimodule
denoted 
$U(\mathfrak{g},e) / U(\mathfrak{g},e)_\sharp$
in \cite[$\S$4.1]{BGK}; the right $U(\mathfrak{g}_0,e)$-module
structure here is defined 
using a homomorphism defined in \cite[Theorem 4.3]{BGK}.
Then given any finite dimensional irreducible 
left $U(\mathfrak{g}_0,e)$-module
$\La$ we can form the Verma module
\begin{equation}\label{ve}
M(\La,e) := U(\mathfrak{g},e) / U(\mathfrak{g},e)_\sharp
\otimes_{U(\mathfrak{g}_0,e)} \La
\end{equation}
as in \cite[$\S$5.2]{BGK}.
As usual, it has a unique irreducible quotient denoted $L(\La,e)$;
see \cite[Theorem 4.5(4)]{BGK}.
On the other hand, in \cite[$\S$4.3]{L1D},
Losev makes a very similar construction of Verma modules, 
but replaces the homomorphism from \cite[Theorem 4.3]{BGK}
with a map constructed in a completely different way
in \cite[(5.6)]{LcatO}.
It is far from clear that Losev's map is
the same as the one in \cite{BGK}, but fortunately 
this has recently been
checked by Brown and Goodwin (in standard Levi type); see \cite[Proposition 3.12]{BrG}.
Hence, as noted in \cite[$\S$3.5]{BrG}, the Verma modules
constructed in \cite{L1D} are the same as the Verma modules $M(\La,e)$
above coming from
\cite{BGK}. This is a crucial point. 

As we are in standard Levi type, i.e.
$e$ is regular in $\mathfrak{g}_0$, we have simply that
$U(\mathfrak{g}_0,e) \cong Z(\mathfrak{g}_0)$, the center of 
$U(\mathfrak{g}_0)$, as goes back to Kostant \cite[$\S$2]{K}.
More precisely, there is a canonical algebra isomorphism 
\begin{equation}\label{ko}
\operatorname{Pr}_0:Z(\mathfrak{g}_0)
\stackrel{\sim}{\rightarrow}
U(\mathfrak{g}_0,e)
\end{equation}
induced by the unique linear projection 
$\Pr_0:U(\mathfrak{g}_0) \twoheadrightarrow U(\mathfrak{b}_0)$
that sends $u (x - \chi(x))$ to zero for each $u \in U(\mathfrak{g}_0)$
and $x \in \mathfrak{m}_0$.
For $\alpha' \in \mathfrak{t}^*$, 
let $L_0'(\alpha')$ denote the irreducible $U(\mathfrak{g}_0)$-module generated
by a $\mathfrak{b}_0'$-highest weight vector of weight $\alpha' - \rho'$.
Let $W_0$ be the subgroup of $W$ consisting of all permutations
such that $\row(i) = \row(w(i))$ for each $1 \leq i \leq N$, which
is the Weyl group of $\mathfrak{g}_0$.
Then we have the Harish-Chandra isomorphism
\begin{equation}\label{HC}
\Psi_0:Z(\mathfrak{g}_0) \stackrel{\sim}{\rightarrow}
S(\mathfrak{t})^{W_0},
\end{equation}
which we normalize so that
$z\in Z(\mathfrak{g}_0)$ acts on $L_0'(\alpha')$
by the scalar $\alpha'(\Psi_0(z))$ for each $\alpha'\in\mathfrak{t}^*$.
Let $\La$ 
be the one dimensional left $U(\mathfrak{g}_0,e)$-module
corresponding under the isomorphisms (\ref{ko}) and (\ref{HC})
to the $S(\mathfrak{t})^{W_0}$-module
$\C_{\rho'(A)}$ of weight $\rho'(A)$.
By the proof of \cite[Theorem 5.5]{BGK}
and \cite[Lemma 5.1]{BGK}, 
we have that
$M(\La,e) \cong M(A,e)$
as left $U(\mathfrak{g},e)$-modules, hence $L(\La,e) \cong L(A,e)$.
So we have identified
$L(A,e)$ with a highest weight module exactly 
as in \cite{L1D}, and our problem (\ref{prob}) 
now becomes to show that
\begin{equation}\label{prob2}
\ann_{U(\mathfrak{g})} S_\chi(L(\Lambda,e))
= 
\ann_{U(\mathfrak{g})} L'(\rho'(A)).
\end{equation}

By the definition of $\La$ and (\ref{HC}), the character of $Z(\mathfrak{g}_0)$
arising from $\Lambda$ via (\ref{ko})
is the same as the central character of $L'_0(\rho'(A))$.
Moreover, by the definition of $\rho'(A)$,
 $L'_0(\rho'(A))$ is an ``anti-dominant'' irreducible 
Verma module, so by \cite[Theorem 8.4.3]{Dix} 
its annihilator in $U(\mathfrak{g}_0)$
is the minimal primitive ideal
generated by the kernel of this central character.
By \cite[Theorem 3.9]{K}, this minimal primitive ideal
is also the annihilator of the
$U(\mathfrak{g}_0)$-module obtained from $\La$ by applying 
the $\mathfrak{g}_0$-version of Skryabin's equivalence.
Now apply \cite[Theorem 5.1.1]{L1D} to deduce (\ref{prob2}).
\end{proof}

Theorem~\ref{labels} has a number of important consequences.
Recalling the definition of the left-justified tableau
$Q(\alpha)$ from the introduction,
let
\begin{equation}
\mathfrak{t}^*_\lambda := \{\alpha \in\mathfrak{t}^*\:|\:
\text{$Q(\alpha)$ has shape $\lambda$}\}.
\end{equation}
For $\alpha \in \mathfrak{t}^*_\lambda$, we define
a $\pi$-tableau
$Q_\pi(\alpha)$ by taking $Q(\alpha)$ 
and sliding the boxes to the right as necessary in order to 
convert it to a $\pi$-tableau. Note $Q_\pi(\alpha)$ is row-equivalent to
a column-strict $\pi$-tableau.

\begin{Lemma}\label{tr}
For any column-strict $\pi$-tableau $A$, 
we have that
$\rho(A)\in\mathfrak{t}^*_\lambda$
and $A \sim Q_\pi(\rho(A))$.
\end{Lemma}

\begin{proof}
This follows easily from the algorithm
to compute $Q(\rho(A))$.
\end{proof}

\begin{Theorem}\label{labels2}
For $\alpha \in \mathfrak{t}_\lambda^*$ we have that
$I(\alpha) = I(L(A,e))$, where $A$ is
any column-strict $\pi$-tableau
with $A \sim Q_\pi(\alpha)$.
\end{Theorem}

\begin{proof}
Lemma~\ref{tr} implies that
$Q_\pi(\rho(A)) \sim A \sim Q_\pi(\alpha)$.
Hence $Q(\rho(A)) \sim Q(\alpha)$, and
we get that $I(\rho(A)) = I(\alpha)$ 
by (\ref{josephs}).
Also by Theorem~\ref{labels} we have that
$I(L(A,e)) = I(\rho(A))$.
Hence $I(\alpha) = I(L(A,e))$.
\end{proof}

The next two corollaries are certainly not new, but still we have
included self-contained proofs in order to illustrate the usefulness of
Theorems~\ref{labels} and \ref{labels2}.
The first
recovers fully the result of
Joseph from \cite[$\S$3.3]{JI}.

\begin{Corollary}[Joseph]\label{jcor}
$\lPrim = \{I(\alpha)\:|\:\alpha \in \mathfrak{t}^*_\lambda\}$.
\end{Corollary}

\begin{proof}
This follows from Theorem~\ref{labels2} and Theorem~\ref{t31},
since we know already by Duflo's theorem and Joseph's irreducibility theorem 
that $\Prim = \{I(\alpha)\:|\:\alpha \in \mathfrak{t}^*\}$ is the disjoint union of the
$\lPrim$'s for all $\la$.
\end{proof}

The next corollary 
is a special case of a 
result proved in arbitrary Cartan type by Losev;
see
\cite[Theorem 1.2.2(viii)]{Lsymplectic} for the surjectivity of the
map
in the statement of the corollary, and
Premet's conjecture formulated in \cite[Conjecture 1.2.1]{Lclass} and proved
in \cite[$\S$4.2]{Lclass} for the injectivity (which 
simplifies in Cartan type $A$ because centralizers are connected).

\begin{Corollary}[Losev]\label{lbij}
The map
\begin{align*}
\left\{
\begin{array}{l}
\text{isomorphism classes of}\\
\text{finite dimensional irreducible}\\
\text{left $U(\mathfrak{g},e)$-modules}
\end{array}
\right\} &\rightarrow \lPrim,
\qquad
[L] \mapsto I(L).
\end{align*}
is a bijection.
\end{Corollary}

\begin{proof}
By Corollary~\ref{jcor}, 
any $I \in \lPrim$ can be represented as $I(\alpha)$
for some $\alpha \in \mathfrak{t}_\lambda^*$.
By Theorem \ref{labels2}, we 
see that $I(\alpha) = I(L)$ for some finite dimensional irreducible
left $U(\mathfrak{g},e)$-module,
hence the map is surjective.
For injectivity, by Theorem~\ref{fdc}, it suffices
to show that $I(L(A,e)) = I(L(B,e))$ implies $A \sim B$
for any column-strict $\pi$-tableaux $A$ and $B$. 
To prove this, use Theorem~\ref{labels} and (\ref{josephs})
to see that
$I(L(A,e))=I(L(B,e))$ implies 
$Q(\rho(A)) \sim Q(\rho(B))$, hence $A \sim B$ by Lemma~\ref{tr}.
\end{proof}

\begin{Remark}\label{brundans}\rm
Let $\operatorname{Prim} U(\mathfrak{g},e)$ denote the space of all
primitive ideals in $U(\mathfrak{g},e)$.
In \cite{Lsymplectic}, Losev shows that there is a well-defined map
\begin{equation*}
?^\dagger:
\operatorname{Prim} U(\mathfrak{g},e) \rightarrow 
\bigcup_{\mu \geq \la}
\operatorname{Prim}_\mu U(\mathfrak{g})
\end{equation*}
such that
$(\ann_{U(\mathfrak{g},e)} M)^\dagger
= I(M)$ for any
irreducible left $U(\mathfrak{g},e)$-module $M$; here $\geq$ is the 
usual dominance ordering on partitions.
Using Theorem~\ref{labels}, Corollary~\ref{jcor} and (\ref{josephs}), it is a purely
combinatorial
exercise to check that this map sends the subset
$$
\operatorname{Prim}_{hw} U(\mathfrak{g},e):= 
\{\ann_{U(\mathfrak{g},e)} L(A,e)\:|\:\text{ for all $\pi$-tableaux
  $A$}\} \subseteq \operatorname{Prim} U(\mathfrak{g},e)
$$
of {\em highest weight} primitive ideals
surjectively onto $\bigcup_{\mu \geq \la} \operatorname{Prim}_\mu
U(\mathfrak{g})$, hence Losev's map $?^\dagger$ is surjective.
We conjecture that it
is also injective (in Cartan type $A$). 
Combined with the preceeding observations,
this conjecture would imply
that $\operatorname{Prim} U(\mathfrak{g},e) = \operatorname{Prim}_{hw}
U(\mathfrak{g},e)$
and moreover
\begin{equation}
\ann_{U(\mathfrak{g},e)} L(A,e)
=
\ann_{U(\mathfrak{g},e)} L(B,e)
\quad\Leftrightarrow\quad
Q(\rho(A)) \sim Q(\rho(B)).
\end{equation}
This would give a classification 
of $\operatorname{Prim}U(\mathfrak{g},e)$
exactly in the spirit of the Duflo-Joseph classification
of $\Prim$ from (\ref{josephs}).
\end{Remark}

Now we turn our attention to deriving some basic properties of the coinvariant
Whittaker functor from (\ref{coinv}).
This functor has its origins in the work of Kostant and Lynch (see e.g. \cite[$\S$3.8]{K} and 
\cite[ch.4]{Ly}) 
though we give a self-contained treatment here.

\begin{Lemma}\label{fin}
The functor $H_0(\mathfrak{m}_\chi,?)$ sends right $U(\mathfrak{g})$-modules that
are finitely generated over $\mathfrak{m}$ to finite dimensional right
$U(\mathfrak{g},e)$-modules.
\end{Lemma}

\begin{proof}
Obvious from the definition (\ref{coinv}).
\end{proof}

\begin{Lemma}\label{bronson}
For any right $U(\mathfrak{p})$-module $V$,
$H_0(\mathfrak{m}_\chi, V \otimes_{U(\mathfrak{p})} U(\mathfrak{g}))$
is isomorphic to the restriction of $V$ to $U(\mathfrak{g},e)$.
\end{Lemma}

\begin{proof}
By the PBW theorem,
$V \otimes_{U(\mathfrak{p})} U(\mathfrak{g}) 
\cong V \otimes U(\mathfrak{m})$ as a right $U(\mathfrak{m})$-module.
It follows easily that the map
$V \rightarrow H_0(\mathfrak{m}_\chi, V \otimes_{U(\mathfrak{p})} U(\mathfrak{g}))$
sending $v$  to the image of $v \otimes 1$ is a vector space isomorphism.
For $u \in U(\mathfrak{g},e)$, this map sends
$vu$ to the image of $vu \otimes 1$, which is the same as
the image of $(v \otimes 1)u$. Hence our map is a homomorphism of
right $U(\mathfrak{g},e)$-modules.
\end{proof}

Given a vector space $M$,
let $M^*$ be the full linear dual
$\hom_{\C}(M,\C)$, and denote the annihilator in $M^*$ of a subspace $N \leq M$
by $N^\circ$ (which is of course canonically isomorphic to 
$(M / N)^*$).
If $M$ is a left module over an associative algebra $A$,
then $M^*$ is naturally a right module with action $(fa)(v) := f(av)$ for $f \in M^*, a \in A$ and $v \in M$.
Similarly if $M$ is a right module then $M^*$ is a left module
with action $(af)(v) = f(va)$.

For a right $U(\mathfrak{m})$-module $M$,
its {\em 
$\mathfrak{m}_\chi$-restricted dual} $M^\#$ is defined from
\begin{equation}\label{hash}
M^\# := \bigcup_{i \geq 0} (M \mathfrak{m}_\chi^i)^\circ
\subseteq M^*.
\end{equation}
This gives a functor $?^\#$ from $\rmod U(\mathfrak{m})$
to vector spaces.

\begin{Lemma}\label{exact}
The functor $?^\#$ is exact.
\end{Lemma}

\begin{proof}
Let $I_\chi$ be the two-sided ideal of $U(\mathfrak{m})$ generated by $\mathfrak{m}_\chi$.
The subspace $(I_\chi^i)^\circ$
of $U(\mathfrak{m})^*$ is naturally a right $U(\mathfrak{m})$-module
with action $(fx)(y) = f(xy)$.
For any right $U(\mathfrak{m})$-module $M$, we claim that the linear map
$$
\theta:\hom_{\mathfrak{m}}(M, (I_\chi^i)^\circ)
\rightarrow (M \mathfrak{m}_\chi^i)^\circ,\quad
f \mapsto \ev \circ f
$$
is an isomorphism, where $\ev:U(\mathfrak{m})^* \rightarrow \C$
is evaluation at $1$.
To see this, take $f \in \hom_{\mathfrak{m}}(M, (I_\chi^i)^\circ)$ and observe that 
$\theta(f)$ annihilates $M \mathfrak{m}_\chi^i$, indeed,
$$
(\ev\circ f)(v x) =
f(vx)(1) = (f(v) x)(1) = f(v)(x) = 0
$$
for $v \in M$ and $x \in \mathfrak{m}_\chi^i$.
Hence the map makes sense.
To prove that it is an isomorphism, construct
a two-sided inverse $\phi:(M \mathfrak{m}_\chi^i)^\circ \rightarrow 
\hom_{\mathfrak{m}}(M, ((I_\chi^i)^\circ)$
by defining
$\phi(g) \in \hom_{\mathfrak{m}}(M, ((I_\chi^i)^\circ)$
for $g \in (M \mathfrak{m}_\chi^i)^\circ$
from
$\phi(g)(v)(u) := g(vu)$
for $v \in M$ and $u \in U(\mathfrak{m})$.

Now let $E_\chi := \bigcup_{i \geq 0} (I_\chi^i)^\circ$,
the space of all $f: U(\mathfrak{m})\rightarrow \C$ which annihilate 
$I_\chi^i$ for sufficiently large $i$.
The result from the previous paragraph taken for all $i$
gives us a natural isomorphism
$$
\hom_{\mathfrak{m}}(M, E_\chi)
=\bigcup_{i \geq 0}
\hom_{\mathfrak{m}}(M, 
 (I_\chi^i)^\circ)
\stackrel{\sim}{\rightarrow}
\bigcup_{i \geq 0} (M \mathfrak{m}_\chi^i)^\circ
= M^\#,
\quad
f \mapsto \ev \circ f
$$
for every right $U(\mathfrak{m})$-module $M$.
Hence the functors $?^\#$ 
and
$\hom_{\mathfrak{m}}(?, E_\chi)$ are isomorphic.
The latter functor is exact because $E_\chi$ is 
an injective right $U(\mathfrak{m})$-module; see 
\cite[Assertion 2]{Skryabin}.
\end{proof}

Now suppose that $M$ is a right $U(\mathfrak{g})$-module.
We observe that the subspace $M^\#$ of $M^*$ from (\ref{hash})
is actually a left $U(\mathfrak{g})$-submodule
belonging to the category $(U(\mathfrak{g}), \mathfrak{m}_\chi)\lmod$.
So we can view $?^\#$ as an exact functor from $\rmod U(\mathfrak{g})$
to $(U(\mathfrak{g}),\mathfrak{m}_\chi)\lmod$.

\begin{Lemma}\label{mainlem}
For any right $U(\mathfrak{g})$-module $M$, we have that
$$
H^0(\mathfrak{m}_\chi, M^\#) = 
H^0(\mathfrak{m}_\chi, M^*) = (M \mathfrak{m}_\chi)^\circ
$$
as subspaces of $M^*$.
Moreover there is a natural isomorphism of left $U(\mathfrak{g},e)$-modules
$(M \mathfrak{m}_\chi)^\circ \cong H_0(\mathfrak{m}_\chi, M)^*.$
\end{Lemma}

\begin{proof}
For the first statement, we observe that
\begin{align*}
H^0(\mathfrak{m}_\chi, M^*) 
&= \{f \in M^*\:|\:xf = 0\text{ for all }x \in \mathfrak{m}_\chi\}\\
&= \{f \in M^*\:|\:(xf)(v) = 0\text{ for all }x \in \mathfrak{m}_\chi, v \in M\}\\
&= \{f \in M^*\:|\:f(vx) = 0\text{ for all }v \in M, x \in \mathfrak{m}_\chi\}\\
&= \{f \in M^*\:|\:f(v) = 0\text{ for all }v \in M\mathfrak{m}_\chi\}
= (M \mathfrak{m}_\chi)^\circ.
\end{align*}
We get that $(M \mathfrak{m}_\chi)^\circ = H^0(\mathfrak{m}_\chi, M^\#)$ too
since there are obviously inclusions
$(M \mathfrak{m}_\chi)^\circ \subseteq H^0(\mathfrak{m}_\chi, M^\#)
\subseteq H^0(\mathfrak{m}_\chi, M^*)$.
Then for the second isomorphism just use the usual natural isomorphism
$(M \mathfrak{m}_\chi)^\circ \cong (M / M \mathfrak{m}_\chi)^*$.
\end{proof}

\begin{Theorem}\label{altdef}
There are natural isomorphisms of right $U(\mathfrak{g},e)$-modules
$$
H^0(\mathfrak{m}_\chi, M^\#)^*\cong 
H_0(\mathfrak{m}_\chi, M)
\cong H^0(\mathfrak{m}_\chi,M^*)^*
$$
for any right $U(\mathfrak{g})$-module that is finitely generated
over $\mathfrak{m}$.
\end{Theorem}

\begin{proof}
Take the duals of the isomorphisms
$$
H^0(\mathfrak{m}_\chi, M^\#) \cong 
H_0(\mathfrak{m}_\chi, M)^*\cong
H^0(\mathfrak{m}_\chi, M^\#)$$ 
from Lemma~\ref{mainlem}
and note 
that
$(H_0(\mathfrak{m}_\chi, M)^*)^* \cong
H_0(\mathfrak{m}_\chi, M)$
by Lemma~\ref{fin}.
\end{proof}

The following corollary is equivalent to \cite[Lemma 4.6]{Ly} (attributed there to N. Wallach).

\begin{Corollary}\label{lynchc}
The functor $H_0(\mathfrak{m}_\chi, ?)$ sends short exact sequences
of right $U(\mathfrak{g})$-modules that are finitely generated
over $\mathfrak{m}$ to short exact sequences of finite dimensional
right $U(\mathfrak{g},e)$-modules.
\end{Corollary}

\begin{proof}
In view of Theorem~\ref{altdef} it suffices to show that
the functor $H^0(\mathfrak{m}_\chi, ?^\#)^*$ is exact.
This is clear
as it is a composition of three exact functors:
the functor $?^\#:U(\mathfrak{g})\lmod \rightarrow (U(\mathfrak{g}),\mathfrak{m}_\chi)\lmod$
which is exact by Lemma~\ref{exact},
then the functor
$H^0(\mathfrak{m}_\chi, ?):(U(\mathfrak{g}),\mathfrak{m}_\chi)\lmod
\rightarrow U(\mathfrak{g},e)\lmod$ which is exact 
as it is an equivalence of categories by Skryabin's theorem, then the
duality $?^*$.
\end{proof}

\section{Irreducible standard modules and induced primitive ideals}\label{sm}

Continuing with our fixed pyramid $\pi$,
we define {\em column-separated} $\pi$-tableaux
in exactly the same way as was done in the introduction 
in the left-justified case.
The following theorem explains the significance of
this notion
from a representation theoretic perspective.
(We point out that there is a typo
in the definition of ``separated'' in \cite{BKrep} in which the inequalities
$r < s$ and $r > s$ are the wrong way round.)

\begin{Theorem}[{\cite[Theorem 8.25]{BKrep}}]\label{sep2}
For a column-strict $\pi$-tableau $A$,
the standard module $V(A,e)$ is irreducible if and only if
$A$ is column-separated, in which case $V(A,e) \cong L(A,e)$.
\end{Theorem}

In the rest of the section we are going to apply this to deduce 
(a slight generalization of) the first equality in Theorem~\ref{sep};
see Theorem~\ref{msup} below.

\begin{Lemma}\label{dizz}
Let $M$ be a right $U(\mathfrak{g})$-module that is free as a 
$U(\mathfrak{m})$-module.
Then
$\ann_{U(\mathfrak{g})} M = \ann_{U(\mathfrak{g})}(M^\#)$,
where $M^\#$ is the left $U(\mathfrak{g})$-module defined 
in the previous section.
\end{Lemma}

\begin{proof}
Take $u \in \ann_{U(\mathfrak{g})} M$
and $f \in M^\#$.
Then $(uf)(v) = f(vu) = 0$ for every $v \in M$,
so $uf = 0$. This shows that
$\ann_{U(\mathfrak{g})}M \subseteq \ann_{U(\mathfrak{g})}(M^\#)$.
Conversely, by the definition (\ref{hash}), we have that
$$
\ann_{U(\mathfrak{g})} (M^\#) = \bigcap_{i \geq 0}
\ann_{U(\mathfrak{g})} (M \mathfrak{m}_\chi^i)^\circ.
$$
So any $u \in \ann_{U(\mathfrak{g})} (M^\#)$
satisfies $f(vu)=(uf)(v) = 0$ for all 
$i \geq 0$,
$f \in (M \mathfrak{m}_\chi^i)^\circ$ and $v\in M$.
This implies for any $v \in M$ that $vu \in M \mathfrak{m}_\chi^i$.
It remains to observe that
$\bigcap_{i \geq 0} M \mathfrak{m}_\chi^i = \bz$.
To see this, it suffices in view of the assumption that $M$ is a free $U(\mathfrak{m})$-module to check that 
$\bigcap_{i \geq 0}
U(\mathfrak{m}) \mathfrak{m}_\chi^i = \bz$.
Twisting by the automorphism of $U(\mathfrak{m})$ sending 
$x \in \mathfrak{m}$ to $x + \chi(x)$,
this is equivalent to the statement
$\bigcap_{i \geq 0}
U(\mathfrak{m}) \mathfrak{m}^i = \bz$, which is easy to see by
considering the (strictly negative) grading on $\mathfrak{m}$.
\end{proof}

\begin{Lemma}
Let $V$ be a finite dimensional left $U(\mathfrak{p})$-module
and $V^*$ be the dual right $U(\mathfrak{p})$-module as in the previous section.
Then
$$
(V^* \otimes_{U(\mathfrak{p})} U(\mathfrak{g}))^\#
\cong S_\chi(V)
$$
as left $U(\mathfrak{g})$-modules. (On the right hand side we
are viewing $V$ as a left $U(\mathfrak{g},e)$-module by the natural restriction.)
\end{Lemma}

\begin{proof}
Both modules belong to the category $(U(\mathfrak{g}),\mathfrak{m}_\chi)\lmod$.
So by Skryabin's equivalence of categories, it suffices to show that
$$
H^0(\mathfrak{m}_\chi,  (V^* \otimes_{U(\mathfrak{p})} U(\mathfrak{g}))^\#)
\cong
V
$$
as left $U(\mathfrak{g},e)$-modules.
By Lemma~\ref{mainlem}, we have that
$$
H^0(\mathfrak{m}_\chi,  (V^* \otimes_{U(\mathfrak{p})} U(\mathfrak{g}))^\#)
\cong 
H_0(\mathfrak{m}_\chi,  V^* \otimes_{U(\mathfrak{p})} U(\mathfrak{g}))^*.
$$
It remains to observe by Lemma~\ref{bronson} 
that $H_0(\mathfrak{m}_\chi,  V^* \otimes_{U(\mathfrak{p})}
U(\mathfrak{g}))
\cong
V^*$,
hence
$H_0(\mathfrak{m}_\chi,  V^* \otimes_{U(\mathfrak{p})} U(\mathfrak{g}))^*
\cong V$ as $V$ is finite dimensional.
\end{proof}

Let $A$ be a column-strict $\pi$-tableau. 
Recall the weight $\gamma(A)$ from (\ref{gammadef}) and the subsequent definition of the standard module $V(A,e)$; it is the restriction of
the left $U(\mathfrak{p})$-module $V(A)$ to the subalgebra $U(\mathfrak{g},e)$.

\begin{Lemma}\label{mainid}
For any column-strict $\pi$-tableau $A$, we have that
\begin{equation}\label{mainidf}
\ann_{U(\mathfrak{g})}
(V(A)^*\otimes_{U(\mathfrak{p})} U(\mathfrak{g}))
=
I(V(A,e)).
\end{equation}
\end{Lemma}

\begin{proof}
This is a consequence of the previous two lemmas
and the definition (\ref{il}).
\end{proof}

It is a bit awkward at this point 
that the module on the left hand side of (\ref{mainidf}) 
is a right module.
We will get around this by twisting with a suitable anti-automorphism,
at the price of a
shift by the special weight $\beta$ from (\ref{betadef})
(and some temporary notational issues).
Observe that $\beta$ extends uniquely to a character of $\mathfrak{p}$.
Let $\C_{\beta}$ be the corresponding one dimensional left
$U(\mathfrak{p})$-module.

We need to 
work momentarily with a different pyramid $\pi^t$
associated to the transpose $\sigma^t$ of the shift matrix $\sigma$;
in other words $\pi^t$ is obtained from $\pi$
by reversing the order of the columns.
For example if
\begin{equation}
\pi=\begin{picture}(39,20)
\put(3,-16){\line(0,1){24}}
\put(15,-16){\line(0,1){36}}
\put(27,-16){\line(0,1){36}}
\put(39,-16){\line(0,1){12}}
\put(3,-16){\line(1,0){36}}
\put(3,-4){\line(1,0){36}}
\put(3,8){\line(1,0){24}}
\put(15,20){\line(1,0){12}}
\put(9,2){\makebox(0,0){$1$}}
\put(9,-10){\makebox(0,0){$2$}}
\put(21,14){\makebox(0,0){$3$}}
\put(21,2){\makebox(0,0){$4$}}
\put(21,-10){\makebox(0,0){$5$}}
\put(33,-10){\makebox(0,0){$6$}}
\end{picture}
\qquad\text{then}\qquad
\pi^t=\begin{picture}(39,20)
\put(3,-16){\line(0,1){12}}
\put(15,-16){\line(0,1){36}}
\put(27,-16){\line(0,1){36}}
\put(39,-16){\line(0,1){24}}
\put(3,-16){\line(1,0){36}}
\put(3,-4){\line(1,0){36}}
\put(15,8){\line(1,0){24}}
\put(27,20){\line(-1,0){12}}
\put(9,-10){\makebox(0,0){$1$}}
\put(21,14){\makebox(0,0){$2$}}
\put(21,2){\makebox(0,0){$3$}}
\put(21,-10){\makebox(0,0){$4$}}
\put(33,-10){\makebox(0,0){$6$}}
\put(33,2){\makebox(0,0){$5$}}
\end{picture}\,.\label{eg}
\end{equation}
\vspace{0.5mm}

\noindent
Let $\mathfrak{p}^t$ (resp.\ $e^t$, resp.\ $U(\mathfrak{g},e^t)$)
be 
defined in the same way as $\mathfrak{p}$ (resp.\ $e$, resp.\ $U(\mathfrak{g},e)$)
but starting from the pyramid $\pi^t$ instead of $\pi$.
If $A$ is any $\pi$-tableau, we 
obtain a $\pi^t$-tableau $A^t$ 
by reversing the order of the columns again. It makes sense to talk about 
$V(A^t)$, 
$V(A^t,e^t)$ and $L(A^t,e^t)$, which are 
$U(\mathfrak{p}^t)$- and $U(\mathfrak{g},e^t)$-modules.

Now
we define the appropriate anti-automorphism.
As usual label 
the boxes of $\pi$ in order down columns starting from the leftmost
column.
Let $i'$ be the
entry in the $i$th box of the tableau obtained by
writing the numbers $1,\dots,N$ into the boxes of $\pi$ working in order down
columns starting from the rightmost column; 
for example, in the situation of (\ref{eg}) we have that
$1'=5,2'=6,3'=2,4'=3,5'=4,6'=1$.
Let $t:U(\mathfrak{g}) \rightarrow U(\mathfrak{g})$
be the anti-automorphism with $t(e_{i,j}) = e_{j', i'}$.
Then we have that $t(e) = e^t$
and $t(\mathfrak{p}) = \mathfrak{p}^t$,
so $t$ restricts to an anti-isomorphism
$t:U(\mathfrak{p}) \stackrel{\sim}{\rightarrow} U(\mathfrak{p}^t)$.

\begin{Lemma}\label{booby1}
Suppose that $A$ is a column-strict $\pi$-tableau,
so that $A^t$ is a column-strict $\pi^t$-tableau.
The pull-back $t^*(V(A^t)^*)$ of the
right $U(\mathfrak{p}^t)$-module $V(A^t)^*$
is a left $U(\mathfrak{p})$-module isomorphic to
$\C_\beta \otimes V(A)$.
Hence we have that
\begin{equation}\label{pv0}
t^*(V(A^t)^* \otimes_{U(\mathfrak{p}^t)} U(\mathfrak{g}))
\cong
U(\mathfrak{g}) \otimes_{U(\mathfrak{p})} (\C_\beta \otimes V(A))
\end{equation}
as left $U(\mathfrak{g})$-modules.
\end{Lemma}

\begin{proof}
Suppose $M$ is a finite dimensional left $U(\mathfrak{p}^t)$-module
$M$ and
 we are given an isomorphism 
of left $U(\mathfrak{p})$-modules
$\theta:K \rightarrow t^*(M^*)$.
Then it is clear that the map
$U(\mathfrak{g}) \otimes_{U(\mathfrak{p})} K
\rightarrow
t^*(M^*\otimes_{U(\mathfrak{p}^t)} U(\mathfrak{g})),
u \otimes v \mapsto \theta(v) \otimes t(u)$
is an isomorphism. So the second part of the lemma follows
from the first part.
The first part is a routine exercise in highest weight theory.
\end{proof}

The module on the right hand side of (\ref{pv0})
is a parabolic Verma module attached to the parabolic $\mathfrak{p}$
in the usual sense.
Let us give it a special name: for a column-strict $\pi$-tableau $A$ we set
\begin{equation}\label{pv}
M(A) := U(\mathfrak{g}) \otimes_{U(\mathfrak{p})} (\C_\beta \otimes V(A)).
\end{equation}
This module has irreducible head
\begin{equation}\label{pv2}
L(A) := M(A) / \operatorname{rad} M(A).
\end{equation}
As $V(A)$ has highest weight $\gamma(A) - \beta - \rho$,
$L(A)$ is the usual irreducible highest weight module 
$L(\gamma(A))$ of highest weight $\gamma(A) - \rho$.

\begin{Theorem}\label{msup}
If $A$ is a column-separated $\pi$-tableau then
$$
I(L(A,e)) = 
\ann_{U(\mathfrak{g})} M(A).
$$
\end{Theorem}

\begin{proof}
We need to work with the finite $W$-algebra $U(\mathfrak{g}, e^t)$, notation
as introduced just before Lemma~\ref{booby1}.
Let $A$ be a column-separated $\pi$-tableau.
Then $A^t$ is a column-connected $\pi^t$-tableau, so
$V(A^t,e^t) \cong L(A^t,e^t)$ by Theorem~\ref{sep2}.
By Lemma~\ref{mainid} (for $\pi^t$ rather than $\pi$) 
we get that
$$
I(L(A^t,e^t)) = \ann_{U(\mathfrak{g})} (V(A^t)^* \otimes_{U(\mathfrak{p}^t)}
U(\mathfrak{g})).
$$
Note that $Q(\rho(A)) \sim Q(\rho(A^t))$
by Lemma~\ref{tr}, hence
$I(L(A,e)) = I(L(A^t,e^t))$ by Theorem~\ref{labels} and (\ref{josephs}).
Also Lemma~\ref{booby1} implies that
$$
\ann_{U(\mathfrak{g})}
(V(A^t)^* \otimes_{U(\mathfrak{p}^t)} U(\mathfrak{g}))
=
t(\ann_{U(\mathfrak{g})}
M(A)).$$
So we have established
that 
$
I(L(A,e)) = 
t(\ann_{U(\mathfrak{g})} M(A))$
or equivalently
$$
t^{-1}(I(L(A,e))) = \ann_{U(\mathfrak{g})} M(A).
$$
It remains to observe for any $I \in \Prim$
that $t^{-1}(I) = I$; this follows 
from \cite[5.2(2)]{Je} on noting that $t^{-1}$ is equal to the usual Chevalley
anti-automorphism up to composing with an inner automorphism.
\end{proof}

\section{Irreducible modules and Whittaker coinvariants}\label{sco}

In this section we recall
the construction of the finite dimensional irreducible
left $U(\mathfrak{g},e)$-modules from \cite[$\S$8.5]{BKrep}
by taking Whittaker coinvariants in certain irreducible 
highest weight modules for $\mathfrak{g}$.
Before we can begin,
we need to modify
the definition (\ref{coinv}),
since we want now to use the coinvariant Whittaker functor
in the context of left modules.
Actually both of 
the definitions (\ref{inv})--(\ref{coinv}) are rather asymmetric with respect
to left and right modules.
The reason for this goes back to the original definition 
of the finite $W$-algebra from
(\ref{fw}): one could just as naturally consider
\begin{equation}\label{fw2}
\overline{U}(\mathfrak{g},e) := \{u \in U(\mathfrak{p})\:|\:
u \mathfrak{m}_\chi \subseteq 
\mathfrak{m}_\chi U(\mathfrak{g})\}.
\end{equation}
We call this the {\em opposite finite $W$-algebra}
since there is an {\em anti-isomorphism}
between $\overline{U}(\mathfrak{g},e)$ and $U(\mathfrak{g},e)$.
More precisely, let $U(\mathfrak{g},-e)$ be defined
exactly as in (\ref{fw}) but with $e$
replaced by $-e$ (hence $\chi$ replaced by $-\chi$).
The antipode $S:U(\mathfrak{g}) \rightarrow 
U(\mathfrak{g})$ sending $x \mapsto -x$ for each $x \in \mathfrak{g}$
obviously 
sends $\overline{U}(\mathfrak{g},e)$ to $U(\mathfrak{g},-e)$, and then
$U(\mathfrak{g},-e)$ is isomorphic to $U(\mathfrak{g},e)$
since $-e$ is conjugate to $e$.
Composing, we get an anti-isomorphism
$\overline{U}(\mathfrak{g},e) \stackrel{\sim}{\rightarrow} U(\mathfrak{g},e)$.

Using this anti-isomorphism, it is rather 
routine to deduce opposite versions of most of the results in $\S$\ref{swhitt}
with $U(\mathfrak{g},e)$ replaced by $\overline{U}(\mathfrak{g},e)$.
For example, the opposite versions
of the functors (\ref{inv})--(\ref{coinv}) are functors
\begin{align}\label{bup}
\overline{H}^0(\mathfrak{m}_\chi,?):\rmod U(\mathfrak{g}) 
&\rightarrow \rmod\overline{U}(\mathfrak{g},e),
&M &\mapsto 
\{v \in M\:|\:v \mathfrak{m}_\chi = \bz\},\\
\overline{H}_0(\mathfrak{m}_\chi,?):U(\mathfrak{g})\lmod 
&\rightarrow \overline{U}(\mathfrak{g},e)\lmod, &M &\mapsto 
M / \mathfrak{m}_\chi M.\label{bdown}
\end{align}
The first of these functors gives an equivalence between
$\rmod (U(\mathfrak{g}),\mathfrak{m}_\chi)$ 
and $\rmod \overline{U}(\mathfrak{g},e)$,
where
$\rmod (U(\mathfrak{g}),\mathfrak{m}_\chi)$ is the full subcategory
of
$\rmod U(\mathfrak{g})$ consisting of all 
modules that are locally nilpotent over $\mathfrak{m}_\chi$ (the opposite version of Skryabin's theorem).
Defining $\#:U(\mathfrak{g})\lmod \rightarrow \rmod (U(\mathfrak{g}),\mathfrak{m}_\chi)$
in the oppposite way to in $\S$\ref{swhitt},
the second of these functors satisfies
\begin{equation}\label{duzz}
\overline{H}_0(\mathfrak{m}_\chi, M)
\cong \overline{H}^0(\mathfrak{m}_\chi, M^\#)^*
\end{equation}
for any left $U(\mathfrak{g})$-module $M$ that is finitely generated over $\mathfrak{m}$ (the opposite version of Theorem~\ref{altdef}).

Less obviously, there is also a canonical
{\em isomorphism} between $U(\mathfrak{g},e)$ and
$\overline{U}(\mathfrak{g},e)$.
To record this, recall that
the weight $\beta$ from (\ref{betadef}) extends uniquely
to a character of $\mathfrak{p}$.
The following theorem was
proved originally (in Cartan type $A$ only) by explicit computation in
\cite[Lemma 3.1]{BKrep}, but we cite instead a more conceptual
proof found subsequently (which is valid in all Cartan types).

\begin{Theorem}[{\cite[Corollary 2.9]{BGK}}]\label{twist}
The automorphisms 
$S_{\pm\beta}:U(\mathfrak{p}) \rightarrow U(\mathfrak{p})$
sending $x \in \mathfrak{p}$ to $x \pm \beta(x)$
restrict to mutually inverse isomorphisms
$$
S_\beta:\overline{U}(\mathfrak{g},e)\stackrel{\sim}{\rightarrow}
U(\mathfrak{g},e),\qquad
S_{-\beta}:U(\mathfrak{g},e)\stackrel{\sim}{\rightarrow}
\overline{U}(\mathfrak{g},e).
$$
\end{Theorem}

We get an isomorphism of categories
$S_{-\beta}^*:\overline{U}(\mathfrak{g},e)\lmod
\rightarrow U(\mathfrak{g},e)\lmod$
by pulling back the action
through $S_{-\beta}$.
Composing with $S_{-\beta}^*$, we will always from now on view the functors
(\ref{bup})--(\ref{bdown}) as functors
\begin{align}\label{bup2}
\overline{H}^0(\mathfrak{m}_\chi,?):\rmod U(\mathfrak{g}) 
&\rightarrow \rmod U(\mathfrak{g},e),\\
\overline{H}_0(\mathfrak{m}_\chi,?):U(\mathfrak{g})\lmod 
&\rightarrow U(\mathfrak{g},e)\lmod.\label{bdown2}
\end{align}
Of course we are abusing notation here, but we won't mention $\overline{U}(\mathfrak{g},e)$ again so there should be 
no confusion.

Now let $\mathcal O_\pi$ be the parabolic category $\mathcal O$
consisting of finitely generated $\mathfrak{g}$-modules that are locally
finite over $\mathfrak{p}$ and semisimple over $\mathfrak{h}$.
The basic objects in $\mathcal O_\pi$ are the parabolic Verma
modules $M(A)$ and their irreducible quotients $L(A)$
from (\ref{pv})--(\ref{pv2}).
Recall that both of these modules are of highest weight $\gamma(A)-\rho$.
 
\begin{Lemma}\label{exactagain}
The restriction of the functor $\overline{H}(\mathfrak{m}_\chi,?)$ to $\mathcal{O}_\pi$ is exact
and it sends modules in $\mathcal{O}_\pi$ to finite dimensional
left $U(\mathfrak{g},e)$-modules.
\end{Lemma}

\begin{proof}
Every module in $\mathcal O_\pi$ has a composition series
with composition factors of the form $L(A)$ for various column-strict 
$\pi$-tableaux $A$. Since $L(A)$ is a quotient of $M(A)$ 
it is clearly finitely generated as an $\mathfrak{m}$-module.
Hence
every object in $\mathcal O_\pi$ is finitely generated over $\mathfrak{m}$
and we are done 
by the opposite version of Corollary~\ref{lynchc}.
\end{proof}

\begin{Lemma}\label{gform}
For a column-strict $\pi$-tableau $A$,
we have that
\begin{equation*}
\overline{H}_0(\mathfrak{m}_\chi, M(A))
\cong V(A,e)
\end{equation*}
as left $U(\mathfrak{g},e)$-modules.
\end{Lemma}

\begin{proof}
By the definition of $M(A)$ and the opposite version of
Lemma~\ref{bronson}, we that
$\overline{H}_0(\mathfrak{m}_\chi, M(A))
\cong S_{-\beta}^*(\C_\beta\otimes V(A,e)) \cong V(A,e)$.
\end{proof}

Call a $\pi$-tableau $A$ {\em semi-standard}
if it is column-strict and $\gamma(A) \in \mathfrak{t}^*_\lambda$,
i.e. $Q(\gamma(A))$ has shape $\lambda$.
In the left-justified case, it is an easy exercise to check that
$A$ is semi-standard if and only if $A$ is both column-strict and
row-standard, which hopefully justifies our choice of language.
In other cases the 
semi-standard $\pi$-tableaux are harder to characterize
from a combinatorial point of view.
For example, here are all the semi-standard $\pi$-tableaux
for one particular $\pi$ with entries $1,2,3,3,4,4$:
\begin{align*}
\phantom{Q_\pi(\gamma())}
A &= \begin{picture}(39,22)
\put(3,-16){\line(1,0){36}}
\put(3,-4){\line(1,0){36}}
\put(3,8){\line(1,0){24}}
\put(15,20){\line(1,0){12}}
\put(3,-16){\line(0,1){24}}
\put(15,-16){\line(0,1){36}}
\put(27,-16){\line(0,1){36}}
\put(39,-16){\line(0,1){12}}
\put(9,-10){\makebox(0,0){$2$}}
\put(21,-10){\makebox(0,0){$1$}}
\put(33,-10){\makebox(0,0){$4$}}
\put(9,2){\makebox(0,0){$3$}}
\put(21,2){\makebox(0,0){$3$}}
\put(21,14){\makebox(0,0){$4$}}
\end{picture}
&
\phantom{Q_\pi(\gamma())}
B &= \begin{picture}(39,22)
\put(3,-16){\line(1,0){36}}
\put(3,-4){\line(1,0){36}}
\put(3,8){\line(1,0){24}}
\put(15,20){\line(1,0){12}}
\put(3,-16){\line(0,1){24}}
\put(15,-16){\line(0,1){36}}
\put(27,-16){\line(0,1){36}}
\put(39,-16){\line(0,1){12}}
\put(9,-10){\makebox(0,0){$2$}}
\put(21,-10){\makebox(0,0){$1$}}
\put(33,-10){\makebox(0,0){$3$}}
\put(9,2){\makebox(0,0){$4$}}
\put(21,2){\makebox(0,0){$3$}}
\put(21,14){\makebox(0,0){$4$}}
\end{picture}
& 
\phantom{Q_\pi(\gamma())}
C &= \begin{picture}(39,22)
\put(3,-16){\line(1,0){36}}
\put(3,-4){\line(1,0){36}}
\put(3,8){\line(1,0){24}}
\put(15,20){\line(1,0){12}}
\put(3,-16){\line(0,1){24}}
\put(15,-16){\line(0,1){36}}
\put(27,-16){\line(0,1){36}}
\put(39,-16){\line(0,1){12}}
\put(9,-10){\makebox(0,0){$1$}}
\put(21,-10){\makebox(0,0){$2$}}
\put(33,-10){\makebox(0,0){$4$}}
\put(9,2){\makebox(0,0){$3$}}
\put(21,2){\makebox(0,0){$3$}}
\put(21,14){\makebox(0,0){$4$}}
\end{picture}\:.
\end{align*}
\vspace{0.5mm}

\noindent
To illustrate the next lemma, we note for 
these that
\begin{align*}
Q_\pi(\gamma(A)) &\sim \begin{picture}(39,22)
\put(3,-16){\line(1,0){36}}
\put(3,-4){\line(1,0){36}}
\put(3,8){\line(1,0){24}}
\put(15,20){\line(1,0){12}}
\put(3,-16){\line(0,1){24}}
\put(15,-16){\line(0,1){36}}
\put(27,-16){\line(0,1){36}}
\put(39,-16){\line(0,1){12}}
\put(9,-10){\makebox(0,0){$3$}}
\put(21,-10){\makebox(0,0){$1$}}
\put(33,-10){\makebox(0,0){$4$}}
\put(9,2){\makebox(0,0){$4$}}
\put(21,2){\makebox(0,0){$2$}}
\put(21,14){\makebox(0,0){$3$}}
\end{picture}
&Q_\pi(\gamma(B)) &\sim \begin{picture}(39,22)
\put(3,-16){\line(1,0){36}}
\put(3,-4){\line(1,0){36}}
\put(3,8){\line(1,0){24}}
\put(15,20){\line(1,0){12}}
\put(3,-16){\line(0,1){24}}
\put(15,-16){\line(0,1){36}}
\put(27,-16){\line(0,1){36}}
\put(39,-16){\line(0,1){12}}
\put(9,-10){\makebox(0,0){$3$}}
\put(21,-10){\makebox(0,0){$1$}}
\put(33,-10){\makebox(0,0){$3$}}
\put(9,2){\makebox(0,0){$4$}}
\put(21,2){\makebox(0,0){$2$}}
\put(21,14){\makebox(0,0){$4$}}
\end{picture}
& Q_\pi(\gamma(C)) &\sim \begin{picture}(39,22)
\put(3,-16){\line(1,0){36}}
\put(3,-4){\line(1,0){36}}
\put(3,8){\line(1,0){24}}
\put(15,20){\line(1,0){12}}
\put(3,-16){\line(0,1){24}}
\put(15,-16){\line(0,1){36}}
\put(27,-16){\line(0,1){36}}
\put(39,-16){\line(0,1){12}}
\put(9,-10){\makebox(0,0){$1$}}
\put(21,-10){\makebox(0,0){$2$}}
\put(33,-10){\makebox(0,0){$4$}}
\put(9,2){\makebox(0,0){$3$}}
\put(21,2){\makebox(0,0){$3$}}
\put(21,14){\makebox(0,0){$4$}}
\end{picture}\:.
\end{align*}

\vspace{4mm}

\noindent
Two semi-standard $\pi$-tableaux $A$ and $B$ are {\em parallel},
denoted $A \parallel B$, if one is obtained from the other by a sequence
of transpositions of
pairs of columns of the same height whose entries lie in different cosets of
$\C$ modulo $\Z$.

\begin{Lemma}\label{rect}
There is a unique map $R$ making the following into a commuting diagram
of bijections:
\begin{align*}
\left\{\begin{array}{l}
\text{parallel classes of}\\\text{semi-standard $\pi$-tableaux}
\end{array}
\right\}\\
&\qquad\lPrim.\\
\left\{
\begin{array}{l}
\text{row-equivalence classes of}\\\text{column-strict $\pi$-tableaux}
\end{array}
\right\} 
\begin{picture}(0,0)
\put(11,44){\makebox(0,0){$\searrow$}}
\put(37,50){\makebox(0,0){$\scriptstyle[A] \mapsto I(\gamma(A))$}}
\put(11,12){\makebox(0,0){$\nearrow$}}
\put(37,6){\makebox(0,0){$\scriptstyle[B] \mapsto I(\rho(B))$}}
\put(-75,26){\makebox(0,0){${\scriptstyle R} {\Big\downarrow}$}}
\end{picture}
\end{align*}
More explicitly, $R$ maps $[A]$ to $[B]$ where $B$ is any column-strict
$\pi$-tableau such that $B \sim Q_\pi(\gamma(A))$.
In the special case that $\pi$ is left-justified
(when
a $\pi$-tableau is semi-standard if and only if it is
both column-strict and row-standard)
the map $R$
is induced by the natural inclusion
of semi-standard $\pi$-tableaux into column-strict $\pi$-tableaux. 
\end{Lemma}

\begin{proof}
In \cite[$\S$4.1]{BKrep},
the following purely combinatorial statement is established:
there is a well-defined bijection $R$ from
parallel classes of semi-standard $\pi$-tableaux to
row-equivalence classes of column-strict $\pi$-tableaux
sending $[A]$ to $[B]$
where $B \sim Q_\pi(\gamma(A))$.
To deduce the first part of the lemma from this,
note for such $A$ and $B$ 
that $B \sim Q_\pi(\rho(B))$ by Lemma~\ref{tr},
hence our bijection $R$
sends $[A]$ to $[B]$ where $Q(\gamma(A)) \sim Q(\rho(B))$.
In view of (\ref{josephs}) we deduce that the 
diagram in the statement of the lemma commutes.
It remains to observe that the top right map in the diagram is already
known to be a bijection, thanks to Corollary~\ref{lbij},
Theorem~\ref{fdc} and Theorem~\ref{labels}.
The last statement of the lemma is clear as
$Q_\pi(\gamma(A)) \sim A$ in case $\pi$ is left-justified
and $A$ is semi-standard.
\end{proof}
 
Now we can state (and slighty extend) the main result 
from \cite[$\S$8.5]{BKrep}
which identifies some of the
$\overline{H}_0(\mathfrak{m}_\chi, L(A))$'s
with $L(B,e)$'s.
The equivalences 
in this theorem
originate in work of Irving \cite{I} and proofs in varying degrees of generality 
can be found in several places in the literature.

\begin{Theorem}\label{bigt}
Let $A$ be a column-strict $\pi$-tableau. 
The following conditions are equivalent:
\begin{itemize}
\item[(1)] $A$ is semi-standard;
\item[(2)] the projective cover of $L(A)$ in $\mathcal{O}_\pi$ is
  self-dual;
\item[(3)] $L(A)$ is isomorphic to a submodule of a parabolic Verma module
in $\mathcal{O}_\pi$;
\item[(4)] $\gkdim L(A) = \dim \mathfrak{m}$,
which is the maximum
Gelfand-Kirillov dimension of any module in $\mathcal{O}_\pi$;
\item[(5)] $\gkdim (U(\mathfrak{g}) / \ann_{U(\mathfrak{g})} L(A)) = 
\dim G\cdot e = 2 \dim \mathfrak{m}$;
\item[(6)]
the associated variety $\VA'(\ann_{U(\mathfrak{g})} L(A))$ is the closure of $
G \cdot e$;
\item[(7)] the module
$\overline{H}_0(\mathfrak{m}_\chi, L(A))$ is non-zero.
\end{itemize}
Assuming these conditions hold, we have that
\begin{equation*}
\overline{H}_0(\mathfrak{m}_\chi, L(A))
\cong L(B,e)
\end{equation*}
where $B$ is a column-strict $\pi$-tableau with $B \sim Q_\pi(\gamma(A))$, 
i.e. $[B]$ is the image of $[A]$ under the bijection from Lemma~\ref{rect}.
\end{Theorem}

\begin{proof}
By (\ref{duzz}) and the first paragraph of the proof of \cite[Lemma 8.20]{BKrep},
the restriction
of the functor
$\overline{H}_0(\mathfrak{m}_\chi, ?)$
to $\mathcal{O}_\pi$
is isomorphic to the restriction of the functor $\mathbb{V}$
defined in \cite[$\S$8.5]{BKrep}.
Given this and assuming just that (1) holds, the existence of  an 
isomorphism
$\overline{H}_0(\mathfrak{m}_\chi, L(A))
\cong L(B,e)$
 follows from \cite[Corollary 8.24]{BKrep}.
In particular 
$\overline{H}_0(\mathfrak{m}_\chi, L(A)) \neq \bz$, establishing
that (1) $\Rightarrow$ (7).
(In fact \cite[Corollary 8.24]{BKrep} also 
proves
(7) $\Rightarrow$ (1) but via an argument that
uses the Kazhdan-Lusztig conjecture; 
we will give an alternative argument shortly avoiding that.)

The equivalence (1) $\Leftrightarrow$ (6)
follows from Corollary~\ref{lbij}, 
since $L(A) \cong L(\gamma(A))$ and by definition 
$A$ is semi-standard if and only if $Q(\gamma(A))$
is of shape $\lambda$.
The equivalence of (4) $\Leftrightarrow$ (5) follows by standard properties
of Gelfand-Kirillov dimension; see \cite[Proposition 2.7]{Joldest}.
We refer to \cite[Theorem 4.8]{BKschur} 
for (1) $\Leftrightarrow$ (2) 
$\Leftrightarrow$ (3) and postpone (4) until the next paragraph.
Note that \cite{BKschur} 
proves a slightly weaker result (integral weights, left-justified $\pi$) 
but the argument there extends.

It remains to check (5) $\Leftrightarrow$ (6) $\Leftarrow$ (7).
We have that
$$
\ann_{U(\mathfrak{g})} L(A)\supseteq 
\ann_{U(\mathfrak{g})} M(A)
= 
\ann_{U(\mathfrak{g})} (M(A)^\#),
$$
using the opposite version of Lemma~\ref{dizz}. 
Hence
$$
\VA'(\ann_{U(\mathfrak{g})} L(A))\subseteq 
\VA'(\ann_{U(\mathfrak{g})} (M(A)^\#)).
$$
Since $\overline{H}_0(\mathfrak{m}_\chi, M(A))^*
\cong \overline{H}^0(\mathfrak{m}_\chi, M(A)^\#)$
by (\ref{duzz}), we see using also Lemma~\ref{gform} 
that
$\overline{H}^0(\mathfrak{m}_\chi, M(A)^\#)$ 
is finite dimensional and non-zero.
Hence we can invoke the opposite version of Theorem~\ref{t31}
to deduce
$\VA'(\ann_{U(\mathfrak{g})} (M(A)^\#)) = \overline{G\cdot e}$.
Hence
$\VA'(\ann_{U(\mathfrak{g})} L(A)) \subseteq
\overline{ G \cdot e }$ and the equivalence of (5) and (6) follows by standard
dimension theory.
Also it is obvious that
$$
\ann_{U(\mathfrak{g})} L(A)\subseteq 
\ann_{U(\mathfrak{g})} (L(A)^\#)
$$
so $$
\overline{G\cdot e}
\supseteq
\VA'(\ann_{U(\mathfrak{g})} L(A))\supseteq 
\VA'(\ann_{U(\mathfrak{g})} (L(A)^\#)).
$$
Finally we repeat the earlier argument with (\ref{duzz}) and
the opposite version of Theorem~\ref{t31} to see that that
$\VA'(\ann_{U(\mathfrak{g})} (L(A)^\#)) =
\overline{G\cdot e}$
assuming (7) holds.
Hence (7) $\Rightarrow$ (6).
\end{proof}

From this, we obtain the following alternative classification of the
finite dimensional irreducible left $U(\mathfrak{g},e)$-modules;
cf. Theorem~\ref{fdc}.

\begin{Corollary}\label{altclass}
As $A$ runs over a set of representatives for the parallel classes of semi-standard
$\pi$-tableaux, the modules $\{\overline{H}_0(\mathfrak{m}_\chi,L(A))\}$
give a complete set of pairwise non-isomorphic irreducible $U(\mathfrak{g},e)$-modules.
\end{Corollary}

\begin{proof}
Combine Theorem~\ref{fdc}, Theorem~\ref{bigt} 
and the bijection in Lemma~\ref{rect}.
\end{proof}

\section{Dimension formulae}\label{sgoldie}

Now we are ready to look more closely at the dimensions
of the finite dimensional irreducible $U(\mathfrak{g},e)$-modules.
We note for column-strict $\pi$-tableaux $A$ and $B$ that
the composition multiplicity $[M(A):L(B)]$ is zero unless $A$
and $B$ have the same {\em content} (multiset of entries),
as follows by central character considerations.
Define $(L(A):M(B)) \in \Z$
from the expansion \begin{equation}\label{Stupid}
[L(A)] = {\sum_B} (L(A):M(B)) [M(B)],
\end{equation}
equality in the Grothendieck group of $\mathcal{O}_\pi$,
where we adopt the convention here and for the rest of the
section
that summation over $B$ always means
summation over 
all column-strict $\pi$-tableaux $B$
having the same content as $A$.
Also define
\begin{equation*}
h_\pi := \prod_{\substack{1 \leq i < j \leq N \\ \col(i) = \col(j)}}
\frac{x_i - x_j}{j-i} \in \C[\mathfrak{t}^*],
\end{equation*}
which is relevant because
the Weyl dimension formula tells us that
\begin{equation}\label{wdf}
\dim V(A,e) = 
\dim V(A) = \dim (\C_\beta \otimes V(A))
=h_\pi(\gamma(A))
\end{equation}
for any column-strict $\pi$-tableau $A$.

\begin{Theorem}\label{mydim}
For any column-strict $\pi$-tableau $A$, we have that
\begin{equation*}
\dim \overline{H}_0(\mathfrak{m}_\chi, L(A))= {\sum_B} (L(A):M(B))
h_\pi(\gamma(B)).
\end{equation*}
Moreover $\dim \overline{H}_0(\mathfrak{m}_\chi, L(A)) = 0$
unless $A$ is semi-standard, in which case
it is equal to $\dim L(B,e)$ where $B$ is any column-strict
$\pi$-tableau with $B \sim Q_\pi(\gamma(A))$.
\end{Theorem}

\begin{proof}
The final statement of the theorem is clear from Theorem~\ref{bigt}.
For the first statement, we know by Lemma~\ref{exactagain} that the functor
$\overline{H}_0(\mathfrak{m}_\chi, ?)$ induces a linear map
between the Grothendieck group of $\mathcal O_\pi$ and the
Grothendieck group of the category of finite dimensional left
$U(\mathfrak{g},e)$-modules.
Applying this map to (\ref{Stupid}) and using Lemma~\ref{gform}
gives the identity
$$
[\overline{H}_0(\mathfrak{m}_\chi, L(A))] = {\sum_B} (L(A):M(B)) [V(B,e)].
$$
The dimension formula follows immediately from
this and (\ref{wdf}).
\end{proof}

In the rest of the section 
we explain how to 
rewrite the sum appearing in 
Theorem~\ref{mydim} in terms of the Kazhdan-Lusztig
polynomials from (\ref{bigkl2}). 
Actually for simplicity we will restrict attention from now on to integral
weights, an assumption which can be justified in several different
ways, one being the following result from 
\cite{BKrep}.

\begin{Theorem}[{\cite[Theorem 7.14]{BKrep}}]\label{dimred}
Suppose $A$ is a column-strict $\pi$-tableau.
Partition the set 
$\{1,\dots,l\}$
into subsets 
$\{i_1 < \cdots < i_k\}$
and $\{j_1 < \cdots < j_{l-k}\}$
in such a way that no entry in any of the columns $i_1,\dots,i_k$ of $A$
is in the same coset of $\C$ modulo $\Z$ as any of the entries in
the columns $j_1,\dots,j_{l-k}$.
Let $A'$ (resp.\ $A''$) 
be the column-strict tableau
consisting just of 
columns $i_1,\dots,i_k$ (resp.\ $j_1,\dots,j_{l-k}$) of $A$
arranged in order from left to right.
Then
$$
\dim L(A,e) = \dim L(A',e') \times \dim L(A'',e'')
$$
where $e'$ and $e''$ are the nilpotent elements
associated to the pyramids of shapes $A'$ and $A''$, respectively.
\end{Theorem}

For an anti-dominant weight $\delta \in P$,
recall from the introduction that $W_\delta$ denotes its stabilizer
and $D_\delta$ is the set of minimal length $W / W_\delta$-coset representatives.
Also let 
\begin{equation}\label{colstab}
W^\pi := \{w \in W\:|\:\col(w(i)) = \col(i)\text{ for all }i=1,\dots,N\},
\end{equation}
the {\em column stabilizer} of our pyramid $\pi$,
and $D^\pi$ denote the set of all {maximal} length $W^\pi \backslash
W$-coset representatives. 

\begin{Lemma}\label{klthm}
For column-strict $\pi$-tableaux $A$ and $B$,
we have that 
$$
(L(A):M(B)) = (L(\gamma(A)):M(\gamma(B))).
$$
If $A$ and $B$ have integer entries 
these numbers can be expressed in terms of Kazhdan-Lusztig polynomials
using (\ref{bigkl2}) and (\ref{klform}).
\end{Lemma}

\begin{proof}
We'll work in the Grothendieck group $[\mathcal O]$ of the full BGG
category $\mathcal O$.
By the Weyl character formula, we have that
$$
[M(B)]=\sum_{x \in W^\pi} (-1)^{\ell(x)} [M(x \gamma(B))].
$$
Substituting this into (\ref{Stupid}) and comparing with
the identity (\ref{first}) for $\alpha = \gamma(A)$,
we get that
\begin{multline*}
{\sum_B}
 \sum_{x \in W^\pi} (-1)^{\ell(x)}
(L(A):M(B)) [M(x \gamma(B))]
=
\sum_{\beta}(L(\gamma(A)):M(\beta))
[M(\beta)].
\end{multline*}
Equating coefficients of $[M(\gamma(B))]$ on both sides gives the conclusion.
\end{proof}

Finally for each $w \in W$ we introduce the {polynomial}
\begin{equation}\label{newgoldie}
p_{w}^\pi
 := 
\sum_{z \in D^\pi} (L(w): M(z)) z^{-1}(h_\pi) \in \C[\mathfrak{t}^*].
\end{equation}
Comparing the following with Theorem~\ref{mydim}
and recalling Corollary~\ref{altclass},
these can be viewed as {\em dimension polynomials}
computing the dimensions of finite dimensional irreducible
$U(\mathfrak{g},e)$-modules in families.

\begin{Theorem}\label{maing}
Let $A$ be a column-strict $\pi$-tableau
such that $\gamma(A) \in W \delta$ for some
anti-dominant $\delta \in P$.
Then
$$
p^\pi_w(\delta)
=
{\sum_B} (L(A):M(B)) h_\pi(\gamma(B))
$$
where
$w = d({\gamma(A)})$
and
the sum is over
all column-strict $\pi$-tableaux $B$
having the same content as $A$.
\end{Theorem}

\begin{proof}
Let $A$ and $\delta$
be fixed as in the statement of the
theorem.
Let $\Tab$ be the set of all $\pi$-tableaux
having the same content as $A$.
Notice that $\gamma$ restricts to a bijection
$\gamma:\Tab \rightarrow W \delta$.
Using this bijection we lift the action of $W$ on $\mathfrak{t}^*$
to an action
on $\Tab$, which is just the natural left action 
of the symmetric group $S_N$ on tableaux
given by place permutation of entries, indexing
entries in order down columns starting
from the leftmost column as usual.
Similarly we view functions in $\C[\mathfrak{t}^*]$ now
as functions on $\Tab$, so 
$x_i(B)$ is just the $i$th entry of $B$.
Let $S \in \Tab$ be the special tableau with $\gamma(S) = \delta$
and write simply $d(B)$ for $d({\gamma(B)})$
for $B \in \Tab$.
We make several routine observations:
\begin{itemize}
\item[(1)] The map 
$\Tab \rightarrow D_\delta, B \mapsto d(B)$ is a bijection
with inverse $x \mapsto x S$.
\item[(2)]
For any $x \in W$, we have that $h_\pi(x S) \neq 0$ if and only if
$xS$ has no repeated entries in any column.
\item[(3)] The set $D^\pi_\delta := D^\pi \cap D_\delta$ is a set
of $(W^\pi, W_\delta)$-coset representatives.
\item[(4)]
Assume $x \in W$ is such that $h_\pi(x S) \neq 0$.
Then we have that $x \in D^\pi$ if and only if $x S$ is column-strict.
\item[(5)]
The restriction of the bijection from (1) is a bijection 
between the set of all column-strict $B \in \Tab$
and the set $\{x \in D_\delta^\pi \:|\:h_\pi(x\delta)\neq 0\}$.
\item[(6)]
For $x \in D^\pi_\delta$ with $h_\pi(x\delta) \neq 0$,
we have that $D^\pi \cap (W^\pi x W_\delta) = x W_\delta$.
\end{itemize}
By Lemma~\ref{klthm} and (\ref{klform}),
then
(5), then (3) and (6), we get
that
\begin{align*}
{\sum_B} (L(A):M(B)) h_\pi(\gamma(B))
&=
{\sum_{B}}
\sum_{y \in W_\delta} 
(L(d(A)):M(d(B) y)) h_\pi(B)\\
&=\sum_{x \in D^\pi_\delta}
\sum_{y \in W_\delta}
(L(d(A)):M(xy)) h_\pi(x \delta)\\
&=
\sum_{z \in D^\pi}
(L(d(A)):M(z)) z^{-1}(h_\pi)(\delta).
\end{align*}
Comparing with (\ref{newgoldie}) this proves the theorem.
\end{proof}

\section{Main results}\label{sproofs}

In this section we prove Theorems~\ref{pti}--\ref{one} exactly as
formulated in the introduction.
We begin with the promised reproof of Premet's theorem.

\begin{proof}[Proof of Premet's Theorem~\ref{pti}]
We recall Joseph's algorithm for computing Goldie ranks
of primitive quotients of $U(\mathfrak{g})$
mentioned already in the introduction.
Let $\mathscr L(M,M)$ denote the space of all $\operatorname{ad}
\mathfrak{g}$-locally finite
maps from a left $U(\mathfrak{g})$-module $M$ to itself.
Joseph established the following statements.
\begin{itemize}
\item[(1)] (\cite[$\S$5.10]{JJI})
For any column-strict $\pi$-tableau $A$
we have that $$
\grk \mathscr L(M(A), M(A)) = h_\pi(\gamma(A)).
$$ 
(To state
Joseph's result in this way we have used 
(\ref{pv}) and (\ref{wdf}).)
\item[(2)] (\cite[$\S$8.1]{Jkos})
The following additivity principle holds:
$$
\grk 
\mathscr L(M(A), M(A)) = \sum_B [M(A):L(B)] \rk(B) 
$$
where $\rk(B):=\grk \mathscr L(L(B), L(B))$ if
$L(B)$ is a module of maximal Gelfand-Kirillov dimension in $\mathcal{O}_\pi$,
and $\rk(B) := 0$ otherwise.
(Again we are using the convention 
that summation over $B$ means
summation over 
all column-strict $\pi$-tableaux $B$
having the same content as $A$.)

\item[(3)] (\cite[$\S$9.1]{Jkos})
For any $\alpha \in \mathfrak{t}^*$,
 $\grk \mathscr L(L(\alpha), L(\alpha)) = \grk
U(\mathfrak{g}) / I(\alpha)$.
\end{itemize}
By (1)--(2) we get that
$h_\pi(\gamma(A))
=
{\sum_B}
[M(A):L(B)] \rk(B)$.
Inverting this gives that
$\rk(A) = \sum_B (L(A):M(B)) h_\pi(\gamma(B)).$
Recall also from (\ref{pv2}) 
that $L(A) \cong L(\gamma(A))$.
So using (3) and the
implication (1)$\Rightarrow$(4) from Theorem~\ref{bigt}
we have established 
that
\begin{equation}\label{josep}
\grk U(\mathfrak{g}) / I(\gamma(A)) =
{\sum_B} (L(A):M(B)) h_\pi(\gamma(B))
\end{equation}
for any semi-standard $\pi$-tableau $A$.
Now take any finite dimensional irreducible left $U(\mathfrak{g},e)$-module $L$.
By Corollary~\ref{altclass}, we may assume $L =
\overline{H}_0(\mathfrak{m}_\chi, L(A))$ for a semi-standard
$\pi$-tableau $A$.
Comparing Theorem~\ref{mydim} with Joseph's formula (\ref{josep}),
we see that
$\dim L = \grk U(\mathfrak{g}) / I(\gamma(A))$.
Finally observe that $I(\gamma(A)) = I(L)$ by 
Lemma~\ref{rect},
Theorem~\ref{labels} and Theorem~\ref{bigt}.
\end{proof}

For the rest of the section we assume that the pyramid $\pi$ is
left-justified, keeping $\lambda$ fixed as before.

\begin{proof}[Proof of Theorem~\ref{mt}]
It suffices to show 
for $\alpha \in \mathfrak{t}^*_\la$
that $\grk U(\mathfrak{g}) / 
I(\alpha) = 1$ if and only if $Q(\alpha)$ is row-equivalent
to a column-connected tableau.
By Theorem~\ref{labels2}, we have that $I(\alpha) = I(L(A,e))$
where $A$
is any column-strict tableau
that is row-equivalent to $Q(\alpha)$.
Hence by Theorem~\ref{pti}, 
we see that $\grk U(\mathfrak{g}) / I(\alpha)
= \dim L(A,e)$.
Now apply Theorem \ref{class}.
\end{proof}

\begin{proof}[Proof of Theorem~\ref{sep}]
We may assume that $\alpha \in \mathfrak{t}^*_\lambda$ and
that $Q(\alpha) \sim A$ for some column-separated
tableau $A$.
By Theorem~\ref{labels2} and Theorem~\ref{msup}, we 
deduce that
$I(\alpha) = I(L(A,e)) = \operatorname{ann}_{U(\mathfrak{g})} M(A)$.
Moreover by Theorem~\ref{pti} and Theorem~\ref{sep2}, we have that
$$
\grk U(\mathfrak{g}) / I(\alpha)
=
\grk U(\mathfrak{g}) / I(L(A,e))
=
\dim L(A,e) = \dim V(A,e) = \dim V(A).
$$
It remains to observe from the definition (\ref{pv}) that 
$M(A) \cong U(\mathfrak{g}) \otimes_{U(\mathfrak{p})} F$
where $F$ is as in the statement of Theorem~\ref{sep},
and also $\dim V(A) = \dim F$ since they are equal up to tensoring by
a one dimensional representation.
\end{proof}

\begin{proof}[Proof of Theorem~\ref{red}]
Take any $\alpha \in \mathfrak{t}_\la^*$
and set $A := Q(\alpha)$.
Then for each $z \in \C$
let $A_z$ be the tableau obtained by erasing all entries of $A$
that are not in $z + \Z$, subtracting $z$
from all remaining entries, and then sliding all
boxes
to the left to get a left-justified tableau with integer entries.
It is clear from the definition of $Q(\alpha)$
that each $A_z$ is a column-strict tableau, indeed, 
$A_z = Q(\alpha_z)$ for $\alpha_z$ as in the statement of Theorem~\ref{red}.
Finally let $e_z$ be the nilpotent in $\mathfrak{g}_z$
associated to the pyramid of the same shape as $A_z$.
Applying Theorem~\ref{dimred} (perhaps several times) we get that
$$
\dim L(A,e) = \prod_{z} \dim L(A_z,e_z)
$$
where the product is over a set of coset representatives for $\C$ modulo $\Z$.
This implies Theorem~\ref{red} thanks to Theorem~\ref{labels2} and
Theorem~\ref{pti}.
\end{proof}

\begin{proof}[Proof of Theorem~\ref{myg}]
We may assume that $w$ is minimal in its left cell
and that $Q(w)$ is of shape $\lambda$.
Take any regular anti-dominant $\delta$ and set $\alpha := w \delta \in
\widehat{C}_w$.
Since the entries of $Q(\alpha)$ satisfy the same 
system of inequalities as the entries of
$Q(w)$,
we see that $Q(\alpha) \sim B$ for a column-separated tableau $B$
which is obtained from $Q(\alpha)$ by permuting entries within rows in
exactly the same way as $A$ is obtained from $Q(w)$.
Theorem~\ref{sep} tells us that 
$\grk U(\mathfrak{g}) /
I(\alpha)$ is the dimension of the
irreducible $\mathfrak{h}$-module of highest weight $\gamma(B) -
\rho$, where $\mathfrak{h}$ is the standard Levi subalgebra
$\mathfrak{gl}_{\lambda_1'}(\C) \oplus 
\mathfrak{gl}_{\lambda_2'}(\C) \oplus\cdots$
and $\la'=(\la'_1\geq\la'_2\geq\cdots)$ 
is the transpose of $\la$.
Using the Weyl dimension formula for $\mathfrak{h}$ we deduce
that
$$
\grk U(\mathfrak{g}) /
I(\alpha)
= h_\la(\gamma(B)).
$$
Using (\ref{minimal}), the definition of $h_\lambda$ from the
statement of Theorem~\ref{foldie}, 
and the assumption that $w$ is minimal in its left cell, 
the right hand side here is the same as
$$
\bigg(\prod_{(i,j)} \frac{x_{w(i)}-x_{w(j)}}{d(i,j)} \bigg)(\gamma(Q(\alpha)))
=\bigg(\prod_{(i,j)} \frac{x_i-x_j}{d(i,j)}\bigg) (\delta)
$$
where the product is over pairs $(i,j)$ as in the statement
of the theorem.
By the definition (\ref{goldiedef}), 
this establishes
that
$p_w$ and $\prod_{(i,j)} (x_i-x_j) / d(i,j)$
take the same values on all regular anti-dominant $\gamma$. The theorem follows by density.
\end{proof}

\begin{proof}[Proof of Joseph's Theorem~\ref{foldie}]
Take any $w \in W$ that is minimal in its left cell, and assume 
that $Q(w)$ has shape $\lambda$.
Take any regular anti-dominant $\delta$. Set $\alpha := w \delta \in
\widehat{C}_w$ and $A := Q(\alpha)$, which is 
a semi-standard tableau of shape $\lambda$.
By (\ref{uc}) and (\ref{minimal}), we have that
$d(\alpha) = w$ and $\gamma(A) = \alpha$.
So Theorems~\ref{mydim} and \ref{maing} give that
$\dim 
\overline{H}_0(\mathfrak{m}_\chi, L(A))= 
p^\pi_w(\delta).$
By
Lemma~\ref{rect},
Theorem~\ref{labels} and Theorem~\ref{bigt}
we know that
$I(\overline{H}_0(\mathfrak{m}_\chi, L(A))) = I(\alpha)$.
Hence by Theorem~\ref{pti} 
we deduce that
$$
\grk U(\mathfrak{g}) / I(\alpha)  = 
p^\pi_w(\delta).
$$
(This equality can also be deduced
without finite $W$-algebras 
using Theorem~\ref{maing} and
Joseph's (\ref{josep}) directly.)
Comparing with
(\ref{goldiedef}) we have therefore shown that
$p_w(\delta) = p^\pi_w(\delta)$
for all $\delta$ in a Zariski dense subset of $\mathfrak{t}^*$,
so $p_w = p^\pi_w$. It remains to observe that the polynomial
$p^\pi_w$ from (\ref{newgoldie}) is the same as the one in on the right hand
side of (\ref{bform}) in the left-justified case.
\end{proof}

\begin{proof}[Proof of Theorem~\ref{one}]
Let $w \in W$ be minimal in its left cell, and assume that $Q(w)$ is of shape
$\lambda$.
Like in the proof of Theorem~\ref{maing}, we use the map $\gamma$ from (\ref{newgamma}) to lift the action of $W$ on $\mathfrak{t}^*$ to an action
on tableaux of shape $\lambda$ by place permutation.
Let $\Tab$ be the set 
of all tableaux of shape $\lambda$ with entries
$\{1,\dots,N\}$ and $S\in \Tab$ be the unique tableau 
with $\gamma(S) = -\rho$. 
We obviously get a bijection 
$W \rightarrow \Tab, w \mapsto w S$.
For any $x \in W$ we have that $x \in D^\lambda$ if and only if 
$x S$ is column-strict, so our bijection 
identifies $D^\lambda$ with the column-strict tableaux in $\Tab$.
Under this identification, it is well known that
the usual Bruhat order $\geq$ on $D^\lambda$ corresponds to the
partial order $\geq$ on column-strict tableaux
such that $A \geq B$ if and only if 
we can
pass from column-strict tableau $A$ to column-strict tableau $B$
by repeatedly applying the following basic move:
\begin{itemize}
\item[(1)]
 find entries $i > j$ in $A$ such that the column containing $i$
is strictly to the left of the column containing $j$;
\item[(2)]
interchange these entries then re-order entries within columns
to obtain another column-strict tableau.
\end{itemize}

Now to prove the result,
let $C$ be the tableau 
from the statement of Theorem~\ref{one}.
Using the explicit formula for $p_w$ from Theorem~\ref{foldie}, we 
need to show that
$$
\sum_{z \in D^\lambda}
(L(w):M(z)) h_\lambda(z w^{-1} \gamma(C))  = 1.
$$
By (\ref{rel}) and (\ref{minimal}) we know that $w S = Q(w)$, which is standard
so certainly column-strict, hence $w \in D^\lambda$.
So there is a term in the above sum with $z=w$, and for this $z$ it is obvious that
$(L(w):M(z)) h_\lambda(z w^{-1}\gamma(C))
= h_\lambda(\gamma(C)) = 1$.
Since $(L(w):M(z)) = 0$ unless $z \leq w$ in the Bruhat order on $W$,
it remains to show that
$h_\lambda(z w^{-1} \gamma(C)) = 0$
for any $z \in D^\lambda$ such that $z < w$.
To see this, take such an element $z$ and let $A := w S$ and $B := z S$,
so $A$ is standard, $B$ is column-strict and 
$A > B$ (in the partial order on column-strict tableau defined in the first
paragraph of the proof).
In the next paragraph, we show that there exist 
$1 \leq i < j \leq N$ such that the numbers $i$ and $j$
appear in the same row of $A$ and in the same column of $B$.
We deduce in the notation from $\S$\ref{s1d}
that
$\row(w(i)) = \row(w(j))$
and $\col(z(i)) = \col(z(j))$.
Hence
$$
(x_{z(i)} - x_{z(j)})(z w^{-1} \gamma(C))
=
(x_{i} - x_{j})(w^{-1} \gamma(C))
=
(x_{w(i)} - x_{w(j)})(\gamma(C)) = 0
$$
and
$x_{z(i)}-x_{z(j)}$ is a linear factor of $h_\lambda$.
This implies that $h_\lambda(z w^{-1} \gamma(C)) = 0$ as required.

It remains to prove the following claim: 
given tableaux $A > B$ of shape $\lambda$ with $A$
standard and $B$ column-strict,
there exist $1 \leq i < j \leq N$ such that $i$ and $j$ appear in the same row
of $A$ and in the same column of $B$.
To see this, let $A_{\leq j}$ (resp.\ $B_{\leq j}$) denote the diagram 
obtained from $A$ (resp.\ $B$) by removing all boxes containing entries
$> j$. 
Choose $1 \leq j \leq N$ so that $A_{\leq (j-1)} = B_{\leq (j-1)}$
but $A_{\leq j} \neq B_{\leq j}$.
Suppose that $j$ appears in column $c$ of $B$,
and observe as $A > B$ that this column is strictly to the left of the
column of $A$ containing $j$.
Suppose also that $j$ appears in row $r$ of $A$,
and observe as $A$ is standard
that this row is strictly below the row of $B$ containing $j$.
As $A_{\leq (j-1)} = B_{\leq (j-1)}$ and $B$ is column-strict,
$A$ and $B$ have the same entry $i \leq j-1$ in row $r$ and column $c$.
Thus the entries $i$ and $j$ lie in the same row $r$ of $A$,
and in the same column $c$ of $B$.
\end{proof}

\end{document}